\documentclass{amsproc}%
\usepackage{amsfonts}
\usepackage{amsmath}
\usepackage{amssymb}
\usepackage{graphicx}%
\theoremstyle{plain}
\newtheorem{theorem}{Theorem}
\newtheorem{lemma}{Lemma}
\newtheorem{corollary}{Corollary}
\newtheorem{proposition}{Proposition}

\theoremstyle{definition}
\newtheorem{definition}{Definition}

\newtheorem{example}{Example}

\theoremstyle{remark}

\numberwithin{equation}{section}
\begin{document}
\title[Nonsymmetric Jack Superpolynomials]{A Superpolynomial Version of Nonsymmetric Jack Polynomials}
\author{Charles F. Dunkl}
\address{Department of Mathematics\\
University of Virginia\\
Charlottesville, VA 22904-4137 }
\email{cfd5z@virginia.edu}
\urladdr{http://people.virginia.edu/\symbol{126}cfd5z/home.html}
\thanks{}
\date{31 August 2020}
\subjclass[2020]{Primary 33C52, 20C30; Secondary 17A70, 81Q80}
\keywords{vector-valued Jack polynomials, supersymmetric polynomials, hook tableaux}
\dedicatory{Dedicated to the memory of Dick Askey, who was my special functions teacher,
and who made it respectable to find exact answers to analysis problems.}
\begin{abstract}
Superpolynomials consist of commuting and anti-commuting variables. By
considering the anti-commuting variables as a module of the symmetric group
the theory of vector-valued nonsymmetric Jack polynomials can be specialized
to superpolynomials. The theory significantly differs from the supersymmetric
Jack polynomials introduced and studied in several papers by Desrosiers,
Mathieu and Lapointe (Nucl. Phys. B606, 2001). The vector-valued Jack
polynomials arise in standard modules of the rational Cherednik algebra and
were originated by Griffeth (T.A.M.S. 362, 2010) for the family $G\left(
n,\ell,N\right)  $ of complex reflection groups. In the present situation
there is an orthogonal basis of anti-commuting polynomials which corresponds
to hook tableaux arising in Young's representations of the symmetric group.
The basis is then used to construct nonsymmetric Jack polynomials by
specializing the machinery set up in a paper by Luque and the author (SIGMA
7,2011). There is an inner product for which these polynomials form an
orthogonal basis, and the squared norms are explicitly found. Supersymmetric
polynomials are obtained as linear combinations of the nonsymmetric Jack
polynomials contained in a submodule; this is based on an idea of Baker and
Forrester (Ann. Comb. 3, 1999). The Poincar\'{e} series for supersymmetric
polynomials graded by degree is obtained and is interpreted in terms of
certain minimal polynomials. There is a brief discussion of antisymmetric
polynomials and an application to wavefunctions of the Calogero-Moser quantum
model on the circle.

\end{abstract}
\maketitle

\section{Introduction}

Superpolynomials involve both commuting and anti-commuting variables. By
interpreting the latter as modules of the symmetric group $\mathcal{S}_{N}$
one can adapt the structure of vector-valued nonsymmetric Jack polynomials to
this setting. Desrosiers, Lapointe, and Mathieu \cite{DLM2001},
\cite{DLM2003a}, \cite{DLM2003}, \cite{DLM2007} constructed Jack
superpolynomials with the use of differential operators motivated by equations
of supersymmetric quantum mechanics. These are significantly different from
the operators to be set up in this work, however the definition of
supersymmetric polynomials is the same. Since our approach uses polynomials
taking values in $\mathcal{S}_{N}$-modules the paper begins with details on
the Young-type construction applied to polynomials in anti-commuting variables
of some fixed degree. Sections \ref{prelim} and \ref{hookrep} contain the
basic definitions and construction of tableau-like polynomials forming
orthogonal bases of the two irreducible modules. Section \ref{Jackpol} is a
brief overview of the vector-valued nonsymmetric Jack polynomials and provides
details of how the group acts on these polynomials. The theory has one free
parameter $\kappa$. As is usual in this area there are numerous arguments
using induction with simple reflections (adjacent transpositions). The
polynomials are a special case of the construction by S. Griffeth \cite{G2010}
which is defined for the family $G\left(  n,p,N\right)  $ of complex
reflection groups. There is a natural inner product for the space of Jack
polynomials for which they are mutually orthogonal. The formula for the
squared norm of a Jack polynomial as a rational function of the parameter
$\kappa$ is stated in this section. The paper of Dunkl and Luque \cite{DL2011}
is used as background reference.

Section \ref{supersym} uses a technique of Baker and Forrester \cite{BF1999}
and ideas from \cite{DL2011} to construct supersymmetric Jack polynomials and
determine the norms. In Section \ref{Pseries} the Poincar\'{e}-Hilbert series
which gives the number of linearly independent supersymmetric polynomials of
given degree is found. The series gives information about the minimal
supersymmetric polynomials which generate all the polynomials over the ring of
ordinary symmetric polynomials. Section \ref{mininorm} concerns certain
minimal polynomials whose norms are found explicitly as polynomials in
$\kappa$. The concluding section \ref{further} briefly discusses the
construction of anti-symmetric polynomials, and the application of
vector-valued Jack polynomials as wavefunctions of the Hamiltonian coming from
the Calogero-Sutherland quantum-mechanical model of identical particles on a
circle with $1/r^{2}$ interactions.

\section{Preliminaries\label{prelim}}

All the subsets appearing here are subsets of $\left\{  1,2,\ldots,N\right\}
$. The complement $E^{C}:=\left\{  j\notin E:1\leq j\leq N\right\}  .$ The
symbol $\#$ denotes the cardinality of a set. Let $\sigma\left(  n\right)
:=\left(  -1\right)  ^{n}$ for $n\in\mathbb{Z}$.

\begin{definition}
For a set $E$ and $1\leq j\leq N$ let%
\begin{align*}
\mathrm{inv}\left(  E\right)   &  =\#\left\{  \left(  i,j\right)  \in E\times
E^{C}:i<j\right\}  ,\\
\mathrm{inv}^{\prime}\left(  E\right)   &  =\#\left\{  \left(  i,j\right)
:i\in E,j\notin E,i<j<N\right\}  ,\\
s\left(  j,E\right)   &  =\#\left\{  i\in E:j<i\right\}  ,\\
E^{\bot}  &  =\left\{  j\in E^{C}:j<N\right\}  .
\end{align*}

\end{definition}

The symmetric group $\mathcal{S}_{N}$ is the group of permutations of
$\left\{  1,2,\ldots,N\right\}  $. For $i\neq j$ the transposition $\left(
i,j\right)  $ fixes $k\neq i,j$ and interchanges $i$ and $j$. The
transpositions $s_{i}:=\left(  i,i+1\right)  $ for $1\leq i<N$ generate
$\mathcal{S}_{N}$, and are fundamental in proofs by induction.

The fermionic variables $\theta_{1},\theta_{2},\ldots,\theta_{N}$ satisfy
$\theta_{i}\theta_{j}=-\theta_{j}\theta_{i}$ for all $i,j$. The symmetric
group $\mathcal{S}_{N}$ acts by permutation: suppose $p\left(  \theta\right)
$ is a polynomial in $\theta_{i}$ and $w\in\mathcal{S}_{N}$ then $wp\left(
\theta\right)  =p\left(  \theta_{w\left(  1\right)  },\theta_{w\left(
2\right)  },\ldots,\theta_{w\left(  N\right)  }\right)  $. As basis elements
for polynomials in $\left\{  \theta_{i}\right\}  $ we use
\[
\phi_{E}:=\theta_{i_{1}}\cdots\theta_{i_{m}},~E=\left\{  i_{1},i_{2}%
,\cdots,i_{m}\right\}  ,1\leq i_{1}<i_{2}<\cdots<i_{m}\leq N.
\]
Our first step is to analyze the representation of $\mathcal{S}_{N}$ acting on
(with $1\leq m\leq N-1$)%
\[
\mathcal{P}_{m}=\mathrm{span}\left\{  \phi_{E}:\#E=m\right\}  .
\]
The space is equipped with the inner product defined by declaring $\left\{
\phi_{E}\right\}  $ to be an orthonormal basis. Clearly $\dim\mathcal{P}%
_{m}=\binom{N}{m}$. We list key properties for the action of $\mathcal{S}_{N}$
on $\left\{  \phi_{E}\right\}  $.

\begin{proposition}
If $j\notin E$ then $\phi_{E}\theta_{j}=\sigma\left(  s\left(  j,E\right)
\right)  \phi_{E\cup\left\{  j\right\}  }$. Suppose $\left(  i,j\right)  $ is
a transposition then (1) $\left(  i,j\right)  \phi_{E}=\phi_{E}$ if $i,j\in
E^{C}$; (2) $\left(  i,j\right)  \phi_{E}=-\phi_{E}$ if $i,j\in E$ ; (3)
$\left(  i,j\right)  \phi_{E}=\sigma\left(  s\left(  i,E\right)  +s\left(
j,E\backslash\left\{  i\right\}  \right)  \right)  \phi_{\left(
E\backslash\left\{  i\right\}  \right)  \cup\left\{  j\right\}  }$ if $i\in
E,j\notin E$.
\end{proposition}

\begin{proof}
Suppose $\left\{  i,j\right\}  \subset E$ and $i<j$, then%
\begin{align*}
\phi_{E}  &  =\sigma\left(  s\left(  j,E\right)  \right)  \phi_{E\backslash
\left\{  j\right\}  }\theta_{j}=\sigma\left(  s\left(  j,E\right)  \right)
\sigma\left(  s\left(  i,E\backslash\left\{  i,j\right\}  \right)  \right)
\phi_{E\backslash\left\{  i,j\right\}  }\theta_{i}\theta_{j},\\
\left(  i,j\right)  \phi_{E}  &  =-\sigma\left(  s\left(  j,E\right)  \right)
\sigma\left(  s\left(  i,E\backslash\left\{  i,j\right\}  \right)  \right)
\phi_{E\backslash\left\{  i,j\right\}  }\theta_{i}\theta_{j}=-\phi_{E}.
\end{align*}
Suppose $i\in E$ and $j\notin E$, then
\begin{align*}
\left(  i,j\right)  \phi_{E}  &  =\left(  i,j\right)  \sigma\left(  s\left(
i,E\right)  \right)  \phi_{E\backslash\left\{  i\right\}  }\theta_{i}\\
&  =\sigma\left(  s\left(  i,E\right)  \right)  \phi_{E\backslash\left\{
i\right\}  }\theta_{j}=\sigma\left(  s\left(  i,E\right)  +s\left(
j,E\backslash\left\{  i\right\}  \right)  \right)  \phi_{\left(
E\backslash\left\{  i\right\}  \right)  \cup\left\{  j\right\}  }.
\end{align*}
Note $s\left(  j,E\backslash\left\{  i\right\}  \right)  =s\left(  j,E\right)
$ if $j>i$ and $s\left(  j,E\backslash\left\{  i\right\}  \right)  =s\left(
j,E\right)  -1$ if $j<i$.
\end{proof}

Denote $\left(  E\backslash\left\{  i\right\}  \right)  \cup\left\{
j\right\}  $ by $\left(  i,j\right)  E$ when $i\in E,j\notin E$. Suppose
$\#\left(  \left\{  i,i+1\right\}  \cap E\right)  =1$, then $s_{i}\phi
_{E}=\phi_{s_{i}E}$ and%
\begin{align*}
\left(  i,i+1\right)   &  \in E\times E^{C}:s_{i}E=\left(  E\backslash\left\{
i\right\}  \right)  \cup\left\{  i+1\right\}  ,\mathrm{inv}\left(
s_{i}E\right)  =\mathrm{inv}\left(  E\right)  -1,\\
\left(  i,i+1\right)   &  \in E^{C}\times E:s_{i}E=\left(  E\backslash\left\{
i+1\right\}  \right)  \cup\left\{  i\right\}  ,\mathrm{inv}\left(
s_{i}E\right)  =\mathrm{inv}\left(  E\right)  +1.
\end{align*}
We set up a duality map from $\mathcal{P}_{m}$ to $\mathcal{P}_{N-m}$.

\begin{definition}
\label{defdual}Suppose $\#E=m$ then define $\delta\phi_{E}:=\sigma\left(
\mathrm{inv}\left(  E\right)  \right)  \phi_{E^{C}}$. Extend $\delta$ to
$\mathcal{P}_{m}$ by linearity, so that $\delta$ is a linear isomorphism from
$\mathcal{P}_{m}$ to $\mathcal{P}_{N-m}$.
\end{definition}

\begin{proposition}
Suppose $1\leq i<N$ then $\delta s_{i}=-s_{i}\delta$.
\end{proposition}

\begin{proof}
Suppose $i,i+1\in E$ then $s_{i}\phi_{E}=-\phi_{E}$ and $s_{i}\phi_{E^{C}%
}=\phi_{E^{C}}$ ; hence $\delta s_{i}\phi_{E}=-s_{i}\delta\phi_{E}$. Suppose
$i,i+1\in E^{C}$ then $s_{i}\phi_{E}=\phi_{E}$ and $s_{i}\phi_{E^{C}}%
=-\phi_{E^{C}}$. If $\left(  i,i+1\right)  \in E\times E^{C}$ then
$s_{i}E=\left(  E\backslash\left\{  i\right\}  \right)  \cup\left\{
i+1\right\}  $, $\mathrm{inv}\left(  s_{i}E\right)  =\mathrm{inv}\left(
E\right)  -1$ and $\left(  s_{i}E\right)  ^{C}=s_{i}E^{C}=\left(
E^{C}\backslash\left\{  i+1\right\}  \right)  \cup\left\{  i\right\}  $. Thus%
\begin{align*}
\delta s_{i}\phi_{E}  &  =\delta\phi_{s_{i}E}=\sigma\left(  \mathrm{inv}%
\left(  s_{i}E\right)  \right)  \phi_{\left(  s_{i}E\right)  ^{C}}%
=\sigma\left(  \mathrm{inv}\left(  s_{i}E\right)  \right)  \phi_{s_{i}E^{C}}\\
&  =\sigma\left(  \mathrm{inv}\left(  E\right)  -1\right)  \phi_{s_{i}E^{C}%
}=-s_{i}\delta\phi_{E}.
\end{align*}
A similar argument applies to the case $\left(  i,i+1\right)  \in E^{C}\times
E$.
\end{proof}

The length of $w\in\mathcal{S}_{N}$ is the number of factors of the shortest
product of $\left\{  s_{i}\right\}  $ required to express $w$ and equals
$\ell\left(  w\right)  :=\#\left\{  \left(  i,j\right)  :i<j,w\left(
i\right)  >w\left(  j\right)  \right\}  $.

\begin{corollary}
Suppose $w\in\mathcal{S}_{N}$ then $\delta w=\sigma\left(  \ell\left(
w\right)  \right)  w\delta$.
\end{corollary}

The map $\delta$ can also be interpreted in the direction $\mathcal{P}%
_{N-m}\rightarrow\mathcal{P}_{m}$, and $\delta^{2}=\sigma\left(  m\left(
N-m\right)  \right)  $ since $\delta^{2}\phi_{E}=\sigma\left(  \mathrm{inv}%
\left(  E^{C}\right)  +\mathrm{inv}\left(  E\right)  \right)  \phi_{E}$ and
$\mathrm{inv}\left(  E^{C}\right)  +\mathrm{inv}\left(  E\right)  =\#\left(
E\times E^{C}\right)  $.

\section{Representation theory for hook tableaux\label{hookrep}}

The irreducible representation $\left(  N-m,1^{m}\right)  $ of $\mathcal{S}%
_{N}$ is realized on the span of reverse standard Young tableaux (RSYT) of
shape $\left(  N-m,1^{m}\right)  $. Each such tableau $T$ has the entries
$N,N-1,\ldots,1$ in decreasing order in row 1 and in column 1. For background
on representations of $\mathcal{S}_{N}$ see James and Kerber \cite{JK1991}. In
the examples and diagrams we will display column 1 as the second row, (to cut
down on blank space); suppose $N=9$ and $m=3$ then the entries of $T$ are
indexed as%
\[%
\begin{bmatrix}
T\left[  1,1\right]  & T\left[  1,2\right]  & T\left[  1,3\right]  & T\left[
1,4\right]  & T\left[  1,5\right]  & T\left[  1,6\right] \\
\circ & T\left[  2,1\right]  & T\left[  3,1\right]  & T\left[  4,1\right]  &
&
\end{bmatrix}
.
\]
Here is a typical tableau:%
\[%
\begin{bmatrix}
9 & 8 & 6 & 4 & 3 & 2\\
\circ & 7 & 5 & 1 &  &
\end{bmatrix}
,
\]
Notice that the entries in column 1 determine the tableau, so in this case
there are $\binom{8}{3}$ different tableaux. The \textit{content }of the entry
at $T\left[  i,j\right]  $ is defined to be $j-i$ (in the hook case the
content values are $-m,1-m,\ldots,0,1,\ldots,N-m-1$). The content vector of
$T$ is denoted $\left[  c\left(  i,T\right)  \right]  _{i=1}^{N}$ where
$c\left(  i,T\right)  $ is the content of the cell containing $i$. For the
above example the content vector is $\left[  -3,5,4,3,-2,2,-1,1,0\right]  $.

\begin{definition}
The Jucys-Murphy elements of $\mathcal{S}_{N}$ are the elements of the group
algebra $\mathbb{Q}\mathcal{S}_{N}$ defined by
\[
\omega_{i}:=\sum_{j=i+1}^{N}\left(  i,j\right)  ,1\leq i\leq N
\]
and they commute pairwise (note $\omega_{N}=0$).
\end{definition}

The Jucys-Murphy elements and the simple reflections satisfy the following
commutation relations:%
\begin{align*}
s_{i}\omega_{j}  &  =\omega_{j}s_{i},~j=1,\ldots,i-1,i+2,\ldots,N,\\
s_{i}\omega_{i}s_{i}  &  =\omega_{i+1}+s_{i},~1\leq i<N.
\end{align*}
The representation of $\mathcal{S}_{N}$ on the span of the RSYT's of a given
shape (partition of $N$) is defined in such a way that $\omega_{i}T=c\left(
i,T\right)  T$ for each $i$. The details of the action are given later.

\subsection{The submodule of isotype $\left(  N-m,1^{m}\right)  $.}

In this subsection we construct elements of $\mathcal{P}_{m}$ which correspond
to RSYT's of shape $\left(  N-m,1^{m}\right)  $ and have the appropriate
eigenvalues for $\left\{  \omega_{i}\right\}  $.

\begin{definition}
For $\#E=m+1$ define a polynomial in $\mathcal{P}_{m}$ by
\[
\psi_{E}=\sum_{j\in E}\sigma\left(  s\left(  j,E\right)  \right)
\phi_{E\backslash\left\{  j\right\}  }.
\]

\end{definition}

The following is the starting point for the construction, the steps of which
are ordered by $\mathrm{inv}\left(  E\right)  $.

\begin{definition}
Let $E_{0}=\left\{  N-m,N-m+1,\ldots,N\right\}  $.
\end{definition}

\begin{theorem}
\label{psi_eigval}If $1\leq i\leq N-m-1$ then $\omega_{i}\psi_{E_{0}}=\left(
N-m-i\right)  \phi_{E_{0}}$ and if $N-m\leq i\leq N$ then $\omega_{i}%
\psi_{E_{0}}=-\left(  N-i\right)  \psi_{E_{0}}$.
\end{theorem}

\begin{proof}
Set $g=\psi_{E_{0}}$ and $g_{i}=\phi_{E_{0}\backslash\left\{  i\right\}  }$
for $N-m\leq i\leq N$; thus $g=\sum_{i=N-m}^{N}\left(  -1\right)  ^{N-i}g_{i}%
$. It is clear that $1\leq i<j<N-m$ implies $\left(  i,j\right)  g=g$. If
$N-m\leq i<N$ then $s_{i}g_{j}=-g_{j}$ when $N-m\leq j<i$ or $i+1<j\leq N$
while $s_{i}g_{i}=g_{i+1}$ and $s_{i}g_{i+1}=g_{i}$. Thus $s_{i}g=-g$ for all
$N-m\leq i<N$. It follows that $\left(  i,j\right)  g=-g$ for all $N-m\leq
i<j\leq N$ ($\left(  i,j\right)  $ is a product of an odd number of $s_{k}$,
with $N-m<k<N$). Thus if $N-m\leq i<N$ then $\omega_{i}g=\sum_{j=i+1}%
^{N}\left(  i,j\right)  g=-\left(  N-i\right)  g$. It remains to consider
$\sum_{j=N-m}^{N}\left(  i,j\right)  g$ for $1\leq i<N-m$. For $N-m\leq
j<k\leq N$ let $g_{ijk}=\phi_{\left(  E_{0}\backslash\left\{  j,k\right\}
\right)  \cup\left[  i\right]  }$. Then $\left(  i,j\right)  g_{k}=\left(
-1\right)  ^{j-N+m}g_{ijk}$ and $\left(  i,k\right)  g_{j}=\left(  -1\right)
^{k-N+m+1}g_{ijk}$, and%
\begin{gather*}
\sum_{j=N-m}^{N}\left(  i,j\right)  g=\sum_{j=N-m}^{N}\sum_{k=N-m}^{N}\left(
-1\right)  ^{N-k}\left(  i,j\right)  g_{k}\\
=\sum_{j=N-m}^{N}\left(  -1\right)  ^{N-j}\left(  i,j\right)  g_{j}%
+\sum_{N-m\leq j<k\leq N}\left\{  \left(  -1\right)  ^{N-k}\left(  i,j\right)
g_{k}+\left(  -1\right)  ^{N-j}\left(  i,k\right)  g_{j}\right\} \\
=\sum_{j=N-m}^{N}\left(  -1\right)  ^{N-j}g_{j}=g,
\end{gather*}
because $\left(  i,j\right)  p_{j}=p_{j}$ and%
\[
\left(  -1\right)  ^{N-k}\left(  i,j\right)  g_{k}+\left(  -1\right)
^{N-j}\left(  i,k\right)  g_{j}=\left(  -1\right)  ^{m}\left\{  \left(
-1\right)  ^{j-k}+\left(  -1\right)  ^{k+1-j}\right\}  g_{ijk}=0.
\]
Finally, if $1\leq i<N-m$ then%
\[
\omega_{i}g=\sum_{j=i+1}^{N-m-1}\left(  i,j\right)  g+\sum_{j=N-m}^{N}\left(
i,j\right)  g=\left\{  \left(  N-m-1-i\right)  +1\right\}  g.
\]
This completes the proof.
\end{proof}

\begin{corollary}
The respective $\left\{  \omega_{i}\right\}  $-eigenvalues of $\psi_{E_{0}}$
agree with the content vector of the tableau $T_{0}$ of shape $\left(
N-m,1^{m}\right)  $ given by $T_{0}\left[  i,1\right]  =N+1-i$ for $1\leq
i\leq m+1$ and $T_{0}\left[  1,j\right]  =N-m+1-j$ for $2\leq j\leq N-m$ . The
polynomial $\psi_{E_{0}}$ is of isotype $\left(  N-m,1^{m}\right)  $ and this
representation is a summand of $\mathcal{P}_{m}$.
\end{corollary}

Next we show that $\mathrm{span}\left\{  \psi_{E}:\#E=m+1\right\}  $ is closed
under $\mathcal{S}_{N}$ and thus is an irreducible module of isotype $\left(
N-m,1^{m}\right)  $.

\begin{proposition}
\label{psi_span}Suppose $\#E=m+1$ and $1\leq i<N:$ (1) if $i,i+1\notin E$ then
$s_{i}\psi_{E}=\psi_{E}$; (2) if $i,i+1\in E$ then $s_{i}\psi_{E}=-\psi_{E}$;
(3) if $\#\left(  \left\{  i,i+1\right\}  \cap E\right)  =1$ then $s_{i}%
\psi_{E}=\psi_{s_{i}E}.$
\end{proposition}

\begin{proof}
If $i,i+1\notin E$ then $s_{i}\phi_{E\backslash\left\{  j\right\}  }%
=\phi_{E\backslash\left\{  j\right\}  }$ for each $j\in E$. If $i,i+1\in E$
then $s\left(  i,E\right)  =s\left(  i+1,E\right)  +1$ and%
\begin{align*}
s_{i}\psi_{E}  &  =\sum_{j\in E,j\neq i,i+1}\sigma\left(  s\left(  j,E\right)
\right)  s_{i}\phi_{E\backslash\left\{  j\right\}  }+\sigma\left(  s\left(
i,E\right)  \right)  s_{i}\phi_{E\backslash\left\{  i\right\}  }+\sigma\left(
s\left(  i+1,E\right)  \right)  s_{i}\phi_{E\backslash\left\{  i+1\right\}
}\\
&  =-\sum_{j\in E,j\neq i,i+1}\sigma\left(  s\left(  j,E\right)  \right)
\phi_{E\backslash\left\{  j\right\}  }+\sigma\left(  s\left(  i,E\right)
\right)  \phi_{E\backslash\left\{  i+1\right\}  }+\sigma\left(  s\left(
i+1,E\right)  \right)  \phi_{E\backslash\left\{  i\right\}  }\\
&  =-\psi_{E}%
\end{align*}
because $\sigma\left(  s\left(  i,E\right)  \right)  =-\sigma\left(  s\left(
i+1,E\right)  \right)  $. Suppose $\left(  i,i+1\right)  \in E\times E^{C}$
then $s_{i}E=\left(  E\backslash\left\{  i\right\}  \right)  \cup\left\{
i+1\right\}  ,$ $s\left(  j,s_{i}E\right)  =s\left(  j,E\right)  $ for all
$j\neq i,i+1$, and $s\left(  i+1,s_{i}E\right)  =s\left(  i,E\right)  $; thus
\begin{align*}
s_{i}\psi_{E}  &  =\sum_{j\in E,j\neq i,}\sigma\left(  s\left(  j,E\right)
\right)  s_{i}\phi_{E\backslash\left\{  j\right\}  }+\sigma\left(  s\left(
i,E\right)  \right)  s_{i}\phi_{E\backslash\left\{  i\right\}  }\\
&  =\sum_{j\in E,j\neq i,}\sigma\left(  s\left(  j,s_{i}E\right)  \right)
\phi_{s_{i}E\backslash\left\{  j\right\}  }+\sigma\left(  s\left(  i,E\right)
\right)  \phi_{E\backslash\left\{  i\right\}  }\\
&  =\psi_{s_{i}E}%
\end{align*}
because $\left(  s_{i}E\right)  \backslash\left\{  i+1\right\}  =E\backslash
\left\{  i\right\}  $. If $\left(  i,i+1\right)  \in E^{C}\times E$ then use
the fact that $s_{i}\left(  s_{i}E\right)  =E$ and $s_{i}^{2}=1$.
\end{proof}

\begin{definition}
Let $\mathcal{P}_{m,0}=\mathrm{span}\left\{  \psi_{E}:\#E=m+1\right\}  $. Let
$\mathcal{E}_{0}=\left\{  E:\#E=m+1,N\in E\right\}  .$
\end{definition}

\begin{theorem}
The subspace $\mathcal{P}_{m,0}$ is an irreducible $\mathcal{S}_{N}$-module
isomorphic to the module $\left(  N-m,1^{m}\right)  $, the span of RSYT's with
this shape.
\end{theorem}

\begin{proof}
By Theorem \ref{psi_eigval} and Proposition \ref{psi_span} $\mathcal{P}_{m,0}$
is the $\mathcal{S}_{N}$-module generated by $\psi_{E_{0}}$, which is of
isotype $\left(  N-m,1^{m}\right)  $.
\end{proof}

We set up a correspondence between sets and hook-tableaux which illuminates
the meaning of content vectors and $\left\{  \omega_{i}\right\}  $-eigenvalues.

\begin{definition}
\label{defYE0}Suppose $E\in\mathcal{E}_{0}$ and $E=\left\{  i_{1},\ldots
,i_{m},i_{m+1}\right\}  $, $E^{C}=\left\{  j_{1},\ldots,j_{N-m-1}\right\}  $
with $i_{1}<i_{2}<\cdots<i_{m+1}=N$ and $j_{1}<j_{2}<\cdots$ then $Y_{E}$ is
the RSYT of shape $\left(  N-m,1^{m}\right)  $ given by $Y_{E}\left[
k,1\right]  =i_{m+2-k}$ for $1\leq k\leq m+1$, and $Y_{E}\left[  1,k\right]
=j_{N-m+1-k}$ for $2\leq k\leq N-m$.
\end{definition}

Example: let $N=8,m=3,E=\left\{  2,5,7,8\right\}  $ then%
\[
Y_{E}=%
\begin{bmatrix}
8 & 6 & 4 & 3 & 1\\
\circ & 7 & 5 & 2 &
\end{bmatrix}
.
\]

We will construct $T_{E}\in\mathcal{P}_{m,0}$ satisfying $\omega_{i}%
T_{E}=c\left(  i,Y_{E}\right)  T_{E}$ for each $i$ and $E\in\mathcal{E}_{0}.$
Here is an explicit definition of the content vector satisfying $\left[
c\left(  i,E\right)  \right]  _{i=1}^{N}=\left[  c\left(  i,Y_{E}\right)
\right]  _{i=1}^{N}$

\begin{definition}
\label{contdef}For $E\in\mathcal{E}_{0}$ the content vector is $\left[
c\left(  i,E\right)  \right]  _{i=1}^{N}$ where $c\left(  i,E\right)
=-s\left(  i,E\right)  $ if $i\in E$ and $c\left(  i,E\right)  =s\left(
i,E^{C}\right)  +1$ if $i\notin E$.
\end{definition}

\begin{theorem}
\label{TEexist}For each $E\in\mathcal{E}_{0}$ there exists $T_{E}%
\in\mathcal{P}_{m,0}$ such that $\omega_{i}T_{E}=c\left(  i,E\right)  T_{E}$.
If $i\notin E$ and $i+1\in E\backslash\left\{  N\right\}  $ then $T_{s_{i}%
E}=s_{i}T_{E}-\left(  c\left(  i,E\right)  -c\left(  i+1,E\right)  \right)
^{-1}T_{E}$.
\end{theorem}

\begin{proof}
We proceed by induction on $\mathrm{inv}\left(  E\right)  $ and start with
$E=E_{0}$ with $\mathrm{inv}\left(  E_{0}\right)  =0$ and $T_{E_{0}}%
=\psi_{E_{0}}$; this is valid by Theorem \ref{psi_eigval}. Suppose $T_{E}$ has
been constructed for every $E\in\mathcal{E}_{0}$ with $\mathrm{inv}\left(
E\right)  \leq n$ and $F\in\mathcal{E}_{0}$ satisfies $\mathrm{inv}\left(
F\right)  =n+1.$Then there exists $i<N-1$ such that $i\notin F$ and $i+1\in
F$, so that $\mathrm{inv}\left(  s_{i}F\right)  =\mathrm{inv}\left(  F\right)
-1=n$. Set $E=s_{i}F$ so that $F=s_{I}E$. Let $b=\left(  c\left(  i,E\right)
-c\left(  i+1,E\right)  \right)  ^{-1}$. and define $T_{F}:=s_{i}T_{E}-bT_{E}%
$. If $j<i$ or $j>i+1$ then $\omega_{j}s_{i}=s_{i}\omega_{j}$ and so
$\omega_{j}T_{s_{i}E}=c\left(  j,E\right)  T_{s_{i}E}$, also $c\left(
j,E\right)  =c\left(  j,F\right)  $. Then%
\begin{align*}
\omega_{i}T_{F}  &  =\omega_{i}\left(  s_{i}T_{E}-bT_{E}\right)  =\left(
1+s_{i}\omega_{i+1}\right)  T_{E}-\omega_{i}bT_{E}\\
&  =T_{E}+c\left(  i+1,E\right)  s_{i}T_{E}-c\left(  i,E\right)  bT_{E}\\
&  =c\left(  i+1,E\right)  \left(  s_{i}T_{E}-bT_{E}\right)  .
\end{align*}
Similarly $\omega_{i+1}T_{F}=c\left(  i,E\right)  T_{F}$. Since $c\left(
i,E\right)  =c\left(  i+1,F\right)  $ and $c\left(  i+1,E\right)  =c\left(
i,F\right)  $ this completes the proof.
\end{proof}

It may seem that there is a uniqueness problem in the construction of
$T_{s_{i}E}$ (possibly $s_{i}E=s_{j}F$) but the eigenvalues of $\left\{
\omega_{i}\right\}  $ do determine $T_{E}$ up to a multiplicative constant.
There is a useful triangularity property for $\left\{  \psi_{E}\right\}  $
which completes the uniqueness proof.

\begin{definition}
For $0\leq n\leq m\left(  N-1-m\right)  $ let $\mathcal{P}_{m,0}^{\left(
n\right)  }=\mathrm{span}\left\{  \psi_{E}:E\in\mathcal{E}_{0},\mathrm{inv}%
\left(  E\right)  \leq n\right\}  $.
\end{definition}

The extreme cases are $\mathrm{inv}\left(  \left\{  N-m,\ldots,N\right\}
\right)  =0$ and $\mathrm{inv}\left(  \left\{  1,2,\ldots,m,N\right\}
\right)  =m\left(  N-1-m\right)  $. There is an important relation to Gaussian
binomial coefficients:%
\[
\sum_{E\in\mathcal{E}_{0}}q^{\mathrm{inv}\left(  E\right)  }=%
\genfrac{[}{]}{0pt}{}{N-1}{m}%
_{q}=\frac{\left(  q;q\right)  _{N-1}}{\left(  q;q\right)  _{m}\left(
q;q\right)  _{N-1-m}},
\]
where $\left(  a;q\right)  _{n}:=\prod\limits_{i=0}^{n-1}\left(
1-aq^{i}\right)  $. There are representation-theoretic meanings of this series
which will be discussed later (Section \ref{Pseries}). For example if
$N=6,m=2$ then $%
\genfrac{[}{]}{0pt}{}{5}{2}%
_{q}=1+q+2q^{2}+2q^{3}+2q^{4}+q^{5}+q^{6}$.

\begin{proposition}
Suppose $E\in\mathcal{E}_{0}$ and $\mathrm{inv}\left(  E\right)  =n\geq1$ then
$T_{E}-\psi_{E}\in\mathcal{P}_{m,0}^{\left(  n-1\right)  }$.
\end{proposition}

\begin{proof}
By Proposition \ref{psi_span} if $\psi_{E}\in\mathcal{P}_{m,0}^{\left(
k\right)  }$ and $i<N-1$ then $s_{i}\psi_{E}\in\mathcal{P}_{m,0}^{\left(
k+1\right)  }$. Proceeding by induction suppose the statement is true for some
$n$ and $\mathrm{inv}\left(  F\right)  =n+1,F\in\mathcal{E}_{0}$; then
$F=s_{i}E$ with $\left(  i,i+1\right)  \in E^{C}\times E$ and $\mathrm{inv}%
\left(  E\right)  =n$. Let $T_{E}=\psi_{E}+p_{E}$ with $p_{E}\in
\mathcal{P}_{m,0}^{\left(  n-1\right)  }$, and $b=\left(  c\left(  i,E\right)
-c\left(  i+1,E\right)  \right)  ^{-1}$ then%
\[
T_{F}=s_{i}T_{E}-bT_{E}=s_{i}\psi_{E}+s_{i}p_{E}+bT_{E},
\]
and $s_{i}p_{E}+bT_{E}\in\mathcal{P}_{m,0}^{\left(  n\right)  }$. By
Proposition \ref{psi_span} (3) $s_{i}\psi_{E}=\psi_{s_{i}E}=\psi_{F}$.
\end{proof}

\begin{corollary}
Suppose $E\in\mathcal{E}_{0}$ and $c\left(  i,E\right)  -c\left(
i+1,E\right)  =\varepsilon$ with $\varepsilon=\pm1$, for some $i<N$, then
$s_{i}T_{E}=\varepsilon T_{E}$.
\end{corollary}

\begin{proof}
Let $p=s_{i}T_{E}-\varepsilon T_{E}$. By an equation similar to those in the
proof of Theorem \ref{psi_eigval} one finds that $\omega_{i}p=c\left(
i+1,E\right)  p$ , $\omega_{i+1}p=c\left(  i,E\right)  p,$ and $\omega
_{j}p=c\left(  j,E\right)  p$ for $j\neq i,i+1$. These eigenvalues are
impossible because $i,i+1$ are in the same row of $Y_{E}$ if $\varepsilon=1$,
or the same column if $\varepsilon=-1$. Thus $p=0.$
\end{proof}

We have shown that $\left\{  \psi_{E}:E\in\mathcal{E}_{0}\right\}  $ is a
basis, implicitly based on the degree of $\left(  N-m,1^{m}\right)  $ being
$\binom{N-1}{m}$. Explicitly it can be shown (with straightforward
computation) that if $\#E=m+1$ and $N\notin E$ then $\psi_{E}=\sum
\limits_{j\in E}\sigma\left(  s\left(  j,E\right)  \right)  \psi_{\left(
j,N\right)  E}$ (recall $\left(  j,N\right)  E=\left(  E\backslash\left\{
j\right\}  \right)  \cup\left\{  N\right\}  $). Also $\left\{  T_{E}%
:E\in\mathcal{E}_{0}\right\}  $ is an orthogonal basis.

\subsection{The submodule of isotype $\left(  N-m+1,1^{m-1}\right)  $}

In order to complete the proof that $\mathcal{P}_{m}\simeq\left(
N-m,1^{m}\right)  \oplus\left(  N-m+1,1^{m-1}\right)  $ we will use the
duality map $\delta$ from $\left(  m,1^{N-m}\right)  $ to $\left(
N-m+1,1^{m-1.}\right)  .$ Begin by transferring the above results to $\left(
m,1^{N-m}\right)  $ by interchanging $m$ and $N-m$. We use $\bot$ to mark the
corresponding objects: $\mathcal{E}_{0}^{\bot}=\left\{  E:\#E=N-m+1,N\in
E\right\}  $, $E_{0}^{\bot}=\left\{  m,m+1,\ldots,N\right\}  $. Thus
$T_{E_{0}}^{\bot}=\psi_{E_{0}}^{\bot}=\sum_{j=m}^{N}\left(  -1\right)
^{N+1-j}\phi_{E_{0}^{\bot}\backslash\left\{  j\right\}  }$ and%
\[
\delta T_{E_{0}}^{\bot}=\sum_{j=m}^{N}\left(  -1\right)  ^{N-j}\sigma\left(
\mathrm{inv}\left(  E_{0}^{\bot}\backslash\left\{  j\right\}  \right)
\right)  \phi_{\left(  E_{0}^{\bot}\backslash\left\{  j\right\}  \right)
^{C}}.
\]
Let $E_{1}=\left(  E_{0}^{\bot}\right)  ^{C}=\left\{  1,2,\ldots,m-1\right\}
$ and $\left(  E_{0}^{\bot}\backslash\left\{  j\right\}  \right)  ^{C}%
=E_{1}\cup\left\{  j\right\}  $ for $m\leq j\leq N$. Also $\mathrm{inv}\left(
E_{0}^{\bot}\backslash\left\{  j\right\}  \right)  =j-m$ and thus%
\[
\delta T_{E_{0}}^{\bot}=\left(  -1\right)  ^{N-m}\sum_{j=m}^{N}\phi_{E_{i}%
\cup\left\{  j\right\}  }.
\]
Since $s_{i}\delta=-\delta s_{i}$ for all $i<N$ we see (from Theorem
\ref{psi_eigval}) that $\omega_{i}\left(  \delta T_{E_{0}}^{\bot}\right)
=-c\left(  i,E_{0}^{\bot}\right)  \delta T_{E_{0}}^{\bot}$ for $1\leq i\leq N$
and $-c\left(  i,E_{0}^{\bot}\right)  =-\left(  m-i\right)  $ for $1\leq i\leq
m-1$ and $-c\left(  i,E_{0}^{\bot}\right)  =\left(  N-i\right)  $ for $m\leq
i\leq N$. Thus the respective $\left\{  \omega_{i}\right\}  $ eigenvalues of
$\delta T_{E_{0}}^{\bot}$ coincide with the content vector of $Y_{E_{0}}%
^{\bot}$. So the $\mathcal{S}_{N}$-module generated by $\delta T_{E_{0}}%
^{\bot}$ is of isotype $\left(  N-m+1,1^{m-1}\right)  $, of degree
$\binom{N-1}{m-1}$ and this establishes the decomposition of $\mathcal{P}_{m}$
(the sum of the degrees of $\left(  N-m,1^{m}\right)  $ and $\left(
N-m+1,1^{m-1}\right)  $ is $\binom{N-1}{m}+\binom{N-1}{m-1}=\binom{N}{m}$).
Motivated by the formula for $\delta T_{E_{0}}^{\bot}$ we make the following:

\begin{definition}
Suppose $\#E=m-1$ then%
\[
\eta_{E}:=\sum_{j\notin E}\sigma\left(  s\left(  j,E\right)  \right)
\phi_{E\cup\left\{  j\right\}  }.
\]

\end{definition}

Thus $\delta T_{E_{0}}^{\bot}=\left(  -1\right)  ^{N-m}\eta_{E_{1}}$. Suppose
$E\in\mathcal{E}_{0}^{\bot}$ then $\#E^{C}=m-1$ and $N\notin E$, accordingly
define $\mathcal{E}_{1}=\left\{  E:\#E=m-1,N\notin E\right\}  $ and
$\mathcal{P}_{m,1}=\mathrm{span}\left\{  \eta_{E}:E\in\mathcal{E}_{1}\right\}
$. The corresponding definition of $Y_{E}$ (Definition \ref{defYE0}) is:

\begin{definition}
\label{defYE1}Suppose $E\in\mathcal{E}_{1}$ and $E=\left\{  i_{1}%
,\ldots,i_{m-1}\right\}  $, $E^{C}=\left\{  j_{1},\ldots,j_{N-m+1}\right\}  $
with $i_{1}<i_{2}<\cdots$ and $j_{1}<j_{2}<\cdots<j_{N-m+1}=N$ then $Y_{E}$ is
the RSYT of shape $\left(  N-m+1,1^{m-1}\right)  $ given by $Y_{E}\left[
k,1\right]  =i_{m+1-k}$ for $2\leq k\leq m$, $Y_{E}\left[  1,k\right]
=j_{N-m+2-k}$ for $1\leq k\leq N-m+1$.
\end{definition}

Example: let $N=9,m=3,E=\left\{  1,2\right\}  $ then%
\[
Y_{E}=%
\begin{bmatrix}
8 & 7 & 6 & 5 & 4 & 3\\
\circ & 2 & 1 &  &  &
\end{bmatrix}
.
\]

The following is used to find the formula relating $\delta\psi_{F}$ to
$\eta_{E}$ for $F\in\mathcal{E}_{0}^{\bot}$ with $F=E^{C}$.

\begin{lemma}
Suppose $j\in E$, then%
\[
\mathrm{inv}\left(  E\backslash\left\{  j\right\}  \right)  +s\left(
j,E\right)  +s\left(  j,E^{C}\right)  =\mathrm{inv}\left(  E\right)  +\left(
\#E-1\right)  .
\]

\end{lemma}

\begin{proof}
Since $\mathrm{inv}\left(  E\right)  =\#\left\{  \left(  i,k\right)  \in
E\times E^{C}:i<k\right\}  $ we see that the set of pairs being counted for
$\mathrm{inv}\left(  E\backslash\left\{  j\right\}  \right)  $ omits $\left\{
\left(  j,k\right)  :k\in E^{C}\right\}  $ and includes $\left\{  \left(
i,j\right)  :i\in E,i<j\right\}  $. The cardinalities of these two sets are
$s\left(  j,E^{C}\right)  $ and $\left(  \#E-1-s\left(  j,E\right)  \right)  $ respectively.
\end{proof}

\begin{proposition}
\label{deltapsi}Suppose $E\in\mathcal{E}_{1}$ then $\delta\psi_{E^{C}}^{\bot
}=\left(  -1\right)  ^{N-m}\sigma\left(  \mathrm{inv}\left(  E^{C}\right)
\right)  \eta_{E}$.
\end{proposition}

\begin{proof}
Let $F=E^{C}$ then%
\begin{align*}
\delta\psi_{F}^{\bot}  &  =\delta\sum_{j\in F}\sigma\left(  s\left(
j,F\right)  \right)  \phi_{F\backslash\left\{  j\right\}  }=\sum_{j\in
E}\sigma\left(  s\left(  j,F\right)  +\mathrm{inv}\left(  F\backslash\left\{
j\right\}  \right)  \right)  \phi_{\left(  F\backslash\left\{  j\right\}
\right)  ^{C}}\\
&  =\sum_{j\notin E}\sigma\left(  -s\left(  j,E\right)  +\mathrm{inv}\left(
F\right)  +\#F-1\right)  \phi_{E\cup\left\{  j\right\}  }\\
&  =\left(  -1\right)  ^{N-m}\sigma\left(  \mathrm{inv}\left(  F\right)
\right)  \sum_{j\notin E}\sigma\left(  s\left(  j,E\right)  \right)
\phi_{E\cup\left\{  j\right\}  }%
\end{align*}
by the Lemma applied to $F$.
\end{proof}

We are ready to define the basis elements of $\mathcal{P}_{m,1}$ corresponding
to the RSYT of shape $\left(  N-m+1,1^{m-1}\right)  $.

\begin{definition}
For $E\in\mathcal{E}_{1}$ let $T_{E}:=\left(  -1\right)  ^{N-m}\sigma\left(
\mathrm{inv}\left(  E^{C}\right)  \right)  \delta T_{E^{c}}^{\bot}$.
\end{definition}

\begin{proposition}
Suppose $E\in\mathcal{E}_{1}$ and $1\leq i\leq N$ then $\omega_{i}%
T_{E}=-c\left(  i,E^{C}\right)  T_{E}$.
\end{proposition}

Thus $T_{E}$ corresponds to the transpose of $Y_{E^{C}}$. The content vector
$\left[  c\left(  i,E\right)  \right]  _{i=1}^{N}$ of $E$ is given by
$c\left(  i,E\right)  =-1-s\left(  i,E\right)  $ if $i\in E$ and $c\left(
i,E\right)  =s\left(  i,E^{C}\right)  $ if $i\notin E$ (so that $c\left(
i,E\right)  =c\left(  i,Y_{E}\right)  $). From Definition \ref{contdef} it
follows that $c\left(  i,E\right)  =-c\left(  i,E^{C}\right)  $ for all $i$.
This subsection concludes with the transformation properties of $T_{E}$
derived from $\delta$.

\begin{proposition}
Suppose $E\in\mathcal{E}_{1}$, $i\in E$ and $i+1\in E^{C}\backslash\left\{
N\right\}  $ then $T_{s_{i}E}=s_{i}T_{E}-\left(  c\left(  i,E\right)
-c\left(  i+1,E\right)  \right)  ^{-1}T_{E}$.
\end{proposition}

\begin{proof}
Set $F=E^{C}$ then $T_{s_{i}F}^{\bot}=s_{i}T_{F}^{\bot}-\left(  c\left(
i,F\right)  -c\left(  i+1,F\right)  \right)  ^{-1}T_{F}^{\bot}$ by Theorem
\ref{TEexist}. Apply $\delta$ to the equation:%
\begin{align*}
\delta T_{s_{i}F}^{\bot}  &  =-s_{i}\delta T_{F}^{\bot}-\left(  c\left(
i,F\right)  -c\left(  i+1,F\right)  \right)  ^{-1}\delta T_{F}^{\bot},\\
\left(  -1\right)  ^{N-m}\sigma\left(  \mathrm{inv}\left(  s_{i}F\right)
\right)  T_{s_{i}E}  &  =\left(  -1\right)  ^{N-m}\sigma\left(  \mathrm{inv}%
\left(  F\right)  \right)  \left(  -s_{i}-\left(  c\left(  i,F\right)
-c\left(  i+1,F\right)  \right)  ^{-1}\right)  T_{E},\\
T_{s_{i}E}  &  =\left(  s_{i}+\left(  c\left(  i,F\right)  -c\left(
i+1,F\right)  \right)  ^{-1}\right)  T_{E},
\end{align*}
because $\mathrm{inv}\left(  s_{i}F\right)  =\mathrm{inv}\left(  F\right)
+1$. Furthermore $c\left(  i,F\right)  -c\left(  i+1,F\right)  =-c\left(
i,E\right)  +c\left(  i+1,E\right)  $.
\end{proof}

Notice that the direction $E\rightarrow s_{i}E$ decreases $\mathrm{inv}\left(
E\right)  $ and $E_{1}$ maximizes $\mathrm{inv}\left(  E\right)  $ in
$\mathcal{E}_{1}$. The set $\left\{  T_{E}:E\in\mathcal{E}_{1}\right\}  $ is
an orthogonal basis for $\mathcal{P}_{m,1}$.

\subsection{Projections}

Define $\partial\theta_{i}$ by $\partial\theta_{i}\phi_{E}=0$ and
$\partial\theta_{i}\left(  \theta_{i}\phi_{E}\right)  =\phi_{E}$ where
$i\notin E$, thus for $i\in E$%
\[
\partial\theta_{i}\phi_{E}=\sigma\left(  \#\left\{  j\in E:j<i\right\}
\right)  \phi_{E\backslash\left\{  i\right\}  }.
\]
Let $\widetilde{\theta}_{i}$ denote the multiplication operator $\phi
_{E}\longmapsto\theta_{i}\phi_{E}=\sigma\left(  \#\left\{  j\in E:j<i\right\}
\right)  \phi_{E\cup\left\{  i\right\}  }$ if $i\notin E$, otherwise
$\widetilde{\theta}_{i}\phi_{E}=0$. If $\#E=k$ then $\delta\partial_{i}%
\phi_{E}=\left(  -1\right)  ^{N-k}\theta_{i}\delta\phi_{E}$ (straightforward
proof). Let $M:=\sum_{i=1}^{N}\widetilde{\theta}_{i}$, $D:=\sum_{i=1}%
^{N}\partial\theta_{i}$. With the usual calculations it can be shown that%
\begin{align*}
MD+DM  &  =N\\
\delta D\phi_{E}  &  =\sigma\left(  \#E^{C}\right)  M\delta\phi_{E}\\
\left(  \ker D\right)  \cap\mathcal{P}_{m}  &  =\mathcal{P}_{m,0}\\
\left(  \ker M\right)  \cap\mathcal{P}_{m}  &  =\mathcal{P}_{m,1}\\
\delta\left(  \ker D\right)   &  \subset\ker M
\end{align*}
Also $\left(  MD\right)  ^{2}=N\left(  MD\right)  ,~\left(  DM\right)
^{2}=N\left(  DM\right)  $, so the eigenvalues of $MD,DM$ are $N,0$. Thus
$\frac{1}{N}DM$, $\frac{1}{N}MD$ are the projections $\mathcal{P}%
_{m}\rightarrow\mathcal{P}_{m,0}$, $\mathcal{P}_{m}\rightarrow\mathcal{P}%
_{m,1}$ respectively.

\subsection{Norms\label{Tnorms}}

Recall that $\left\{  \phi_{E}:\#E=m\right\}  $ is an orthonormal basis for
$\mathcal{P}_{m}$, and each $s_{i}$ is a self-adjoint isometry (This can be
implemented by defining dual variables $\widehat{\theta}_{i}$ with the
property $\widehat{\theta}_{i}\theta_{j}=\delta_{ij}$ , and if $E=\left\{
i_{1},\ldots,i_{k}\right\}  $ with $i_{1}<i_{2}<\cdots<i_{k}$ then define
$\widehat{\phi_{E}}=\widehat{\theta}_{i_{k}}\cdots\widehat{\theta}_{i_{2}%
}\widehat{\theta}_{i_{1}}$; thus $\left\langle \phi_{F},\phi_{E}\right\rangle
=\widehat{\phi_{F}}\phi_{E}$.) Hence each $\omega_{i}$ is self-adjoint and two
eigenvectors with at least one different $\left\{  \omega_{i}\right\}
$-eigenvalue are orthogonal to each other. Suppose for some $f\in
\mathcal{P}_{m}$ there are $i,b$ such that $f^{\prime}=s_{i}f-bf$ satisfies
$\left\langle f,f^{\prime}\right\rangle =0$ then $\left\langle s_{i}%
f,f\right\rangle =b\left\vert f\right\vert ^{2}$ and $\left\vert f^{\prime
}\right\vert ^{2}=\left\vert s_{i}f\right\vert ^{2}-2b\left\langle
s_{i}f,f\right\rangle +b^{2}\left\vert f\right\vert ^{2}=\left(
1-b^{2}\right)  \left\vert f\right\vert ^{2}$ since $\left\vert s_{i}%
f\right\vert ^{2}=\left\vert f\right\vert ^{2}$. In the context of the
previous subsections $\left\vert T_{s_{i}E}\right\vert ^{2}=\left(
1-b^{2}\right)  \left\vert T_{E}\right\vert ^{2}$. By definition $\left\vert
T_{E_{0}}\right\vert ^{2}=m+1$ and $\left\vert T_{E_{1}}\right\vert
^{2}=N-m+1$.

\begin{proposition}
Suppose $E\in\mathcal{E}_{0}$ then%
\[
\left\vert T_{E}\right\vert ^{2}=\left(  m+1\right)  \prod
\limits_{\substack{1\leq i<j<N\\\left(  i,j\right)  \in E\times E^{C}}}\left(
1-\frac{1}{\left(  c\left(  i,E\right)  -c\left(  j,E\right)  \right)  ^{2}%
}\right)  .
\]

\end{proposition}

\begin{proof}
Argue implicitly by induction on $\mathrm{inv}\left(  E\right)  $. The product
for $E=E_{0}$ is empty ($=1$). For each $E$ let $h\left(  i,j;E\right)
=1-\left(  \left(  c(i,E\right)  -c\left(  i+1,E\right)  \right)  ^{-2}$.
Suppose $\left(  k,k+1\right)  \in E^{C}\times E$, then $\mathrm{inv}\left(
s_{k}E\right)  =\mathrm{inv}\left(  E\right)  +1$, and $h\left(  i,j;E\right)
=h\left(  i,j;s_{k}E\right)  $ for $i,j\neq k,k+1$. Also $h\left(
i,k+1;E\right)  =h\left(  i,k;s_{k}E\right)  $ for $i<k$ and $i\in E$;
$h\left(  k+1,j;E\right)  =h\left(  k,j;s_{k}E\right)  $ for $j>k+1$ and
$j\notin E$. $\left\vert T_{s_{k}E}\right\vert ^{2}$ differs from $\left\vert
T_{E}\right\vert ^{2}$ by the extra factor $h\left(  k,k+1;s_{k}E\right)  $
($=h\left(  k,k+1,E\right)  $)
\end{proof}

\begin{proposition}
Suppose $E\in\mathcal{E}_{1}$ then%
\[
\left\vert T_{E}\right\vert ^{2}=\left(  N-m+1\right)  \prod
\limits_{\substack{1\leq i<j<N\\\left(  i,j\right)  \in E^{C}\times E}}\left(
1-\frac{1}{\left(  c\left(  i,E\right)  -c\left(  j,E\right)  \right)  ^{2}%
}\right)  .
\]

\end{proposition}

\begin{proof}
This follows from the duality map $\delta$ being an isometry and using the
previous formula for $E^{C}\in\mathcal{E}_{0}^{\bot}$.
\end{proof}

\section{Superpolynomials and Nonsymmetric Jack Polynomials\label{Jackpol}}

Here we extend the polynomials in $\left\{  \theta_{i}\right\}  $ by adjoining
$N$ commuting variables $x_{1},\ldots,x_{N}$ (that is $\left[  x_{i}%
,x_{j}\right]  =0,\left[  x_{i},\theta_{j}\right]  =0,\theta_{i}\theta
_{j}=-\theta_{j}\theta_{i}$ for all $i,j$). Each polynomial is a sum of
monomials $x^{\alpha}\phi_{E}$ where $E\subset\left\{  1,2,\ldots,N\right\}  $
and $\alpha\in\mathbb{N}_{0}^{N},x^{\alpha}=\prod\limits_{i=1}^{N}%
x_{i}^{\alpha_{i}}$. The partitions in $\mathbb{N}_{0}^{N}$ are denoted by
$\mathbb{N}_{0}^{N,+}$ ($\lambda\in\mathbb{N}_{0}^{N,+}$ if and only if
$\lambda_{1}\geq\lambda_{2}\geq\ldots\geq\lambda_{N}$). The \textit{fermionic}
degree of this monomial is $\#E$ and the \textit{bosonic} degree is
$\left\vert \alpha\right\vert =\sum_{i=1}^{N}\alpha_{i}$. Let $s\mathcal{P}%
_{m}=\mathrm{span}\left\{  x^{\alpha}\phi_{E}:\alpha\in\mathbb{N}_{0}%
^{N},\#E=m\right\}  $. Then using the decomposition $\mathcal{P}%
_{m}=\mathcal{P}_{m,0}\oplus\mathcal{P}_{m,1}$ let%
\begin{align*}
s\mathcal{P}_{m,0}  &  =\mathrm{span}\left\{  x^{\alpha}\psi_{E}:\alpha
\in\mathbb{N}_{0}^{N},E\in\mathcal{E}_{0}\right\}  ,\\
s\mathcal{P}_{m,1}  &  =\mathrm{span}\left\{  x^{\alpha}\eta_{E}:\alpha
\in\mathbb{N}_{0}^{N},E\in\mathcal{E}_{1}\right\}  .
\end{align*}
The symmetric group $\mathcal{S}_{N}$ acts on $s\mathcal{P}_{m}$ by $wp\left(
x,\theta\right)  =p\left(  xw,\theta w\right)  $ (recall $\left(  xw\right)
_{i}=x_{w\left(  i\right)  }$ and $\left(  \theta w\right)  _{i}%
=\theta_{w\left(  i\right)  }$ for $1\leq i\leq N$).

\begin{definition}
The \textit{Dunkl} and \textit{Cherednik-Dunkl} operators are ($1\leq i\leq
N,p\in s\mathcal{P}_{m}$)
\begin{align*}
\mathcal{D}_{i}p\left(  x;\theta\right)   &  :=\frac{\partial p\left(
x;\theta\right)  }{\partial x_{i}}+\kappa\sum_{j\neq i}\frac{p\left(
x;\theta\left(  i,j\right)  \right)  -p\left(  x\left(  i,j\right)
;\theta\left(  i,j\right)  \right)  }{x_{i}-x_{j}}\text{,}\\
\mathcal{U}_{i}p\left(  x;\theta\right)   &  :=\mathcal{D}_{i}\left(
x_{i}p\left(  x;\theta\right)  \right)  -\kappa\sum_{j=1}^{i-1}p\left(
x\left(  i,j\right)  ;\theta\left(  i,j\right)  \right)  \text{.}%
\end{align*}

\end{definition}

The same commutation relations as for the scalar case hold, that is,%
\begin{align*}
\mathcal{D}_{i}\mathcal{D}_{j}  &  =\mathcal{D}_{j}\mathcal{D}_{i}%
,~\mathcal{U}_{i}\mathcal{U}_{j}=\mathcal{U}_{j}\mathcal{U}_{i},~1\leq i,j\leq
N\\
w\mathcal{D}_{i}  &  =\mathcal{D}_{w\left(  i\right)  }w,\forall
w\in\mathcal{S}_{N};~s_{j}\mathcal{U}_{i}=\mathcal{U}_{i}s_{j},~j\neq i-1,i;\\
s_{i}\mathcal{U}_{i}s_{i}  &  =\mathcal{U}_{i+1}+\kappa s_{i},~\mathcal{U}%
_{i}s_{i}=s_{i}\mathcal{U}_{i+1}+\kappa,~\mathcal{U}_{i+1}s_{i}=s_{i}%
\mathcal{U}_{i}-\kappa\text{.}%
\end{align*}
The simultaneous eigenfunctions of $\left\{  \mathcal{U}_{i}\right\}  $ are
called (vector-valued) nonsymmetric Jack polynomials (NSJP). They are the
special case for the $\mathcal{S}_{N}$-representations $\left(  N-m,1^{m}%
\right)  $ and $\left(  N-m+1,1^{m-1}\right)  $ of the polynomials constructed
by Griffeth \cite{G2010} for the complex reflection groups $G\left(
n.p.N\right)  $. For generic $\kappa$ these eigenfunctions form a basis of
$s\mathcal{P}_{m}$ (\textit{generic} means that $\kappa\neq\frac{m}{n}$ where
$m,n\in\mathbb{Z}$ and $1\leq n\leq N$). They have a triangularity property
with respect to the partial order $\rhd$ on compositions, which is derived
from the dominance order:
\begin{align*}
\alpha &  \prec\beta~\Longleftrightarrow\sum_{j=1}^{i}\alpha_{j}\leq\sum
_{j=1}^{i}\beta_{j},~1\leq i\leq N,~\alpha\neq\beta\text{,}\\
\alpha\lhd\beta &  \Longleftrightarrow\left(  \left\vert \alpha\right\vert
=\left\vert \beta\right\vert \right)  \wedge\left[  \left(  \alpha^{+}%
\prec\beta^{+}\right)  \vee\left(  \alpha^{+}=\beta^{+}\wedge\alpha\prec
\beta\right)  \right]  \text{.}%
\end{align*}
The rank function on compositions is involved in the formula for an NSJP.

\begin{definition}
For $\alpha\in\mathbb{N}_{0}^{N},1\leq i\leq N$%
\[
r_{\alpha}\left(  i\right)  =\#\left\{  j:\alpha_{j}>\alpha_{i}\right\}
+\#\left\{  j:1\leq j\leq i,\alpha_{j}=\alpha_{i}\right\}  \text{,}%
\]
then $r_{\alpha}\in\mathcal{S}_{N}\text{.}$
\end{definition}

A consequence is that $r_{\alpha}\alpha=\alpha^{+}$, the \textit{nonincreasing
rearrangement} of $\alpha$, for any $\alpha\in\mathbb{N}_{0}^{N}$ . For
example if $\alpha=\left(  1,2,0,5,4,5\right)  $ then $r_{\alpha}=\left[
5,4,6,1,3,2\right]  $ and $r_{\alpha}\alpha=\alpha^{+}=\left(
5,5,4,2,1,0\right)  $ (recall $w\alpha_{i}=\alpha_{w^{-1}\left(  i\right)  }$
). Also $r_{\alpha}=I$ if and only if $\alpha\in\mathbb{N}_{0}^{N,+}$.

For each $\alpha\in\mathbb{N}_{0}^{N}$ and $E\in\mathcal{E}_{0}$ or
$E\in\mathcal{E}_{1}$ there is a NSJP $J_{\alpha,E}$ with leading term
$x^{\alpha}\left(  r_{\alpha}^{-1}T_{E}\right)  $, that is,
\begin{equation}
J_{\alpha,E}\left(  x;\theta\right)  =x^{\alpha}\left(  r_{\alpha}^{-1}%
T_{E}\right)  +\sum_{\alpha\rhd\beta}x^{\beta}v_{\alpha,\beta,T}\left(
\kappa;\theta\right)  , \label{dNJP}%
\end{equation}
where $v_{\alpha,\beta,T}\left(  \kappa;\theta\right)  \in\mathcal{P}_{m,0}$
or $\mathcal{P}_{m,1}$, respectively. The coefficients of the polynomials
$v_{\alpha,\beta,T}\left(  \kappa;\theta\right)  $ are rational functions of
$\kappa$. Note $r_{\alpha}^{-1}T_{E}\left(  \theta\right)  =T_{E}\left(
\theta r_{\alpha}^{-1}\right)  $. These polynomials satisfy%
\begin{align*}
\mathcal{U}_{i}J_{\alpha,E}  &  =\zeta_{\alpha,E}\left(  i\right)
J_{\alpha,E},\\
\zeta_{\alpha,E}\left(  i\right)   &  :=\alpha_{i}+1+\kappa c\left(
r_{\alpha}\left(  i\right)  ,E\right)  ,~1\leq i\leq N.
\end{align*}
For detailed proofs see \cite{DL2011}.

\begin{example}
$N=4,m=2,~\alpha=\left(  0,1,1,0\right)  ,~E=\{2,3,4\}\in\mathcal{E}_{0},$
$\left[  c\left(  j,E\right)  \right]  _{j=1}^{4}=\left[  1,-2,-1,0\right]  $%
\begin{gather*}
J_{\alpha,E}=\left(  x_{2}x_{3}-\frac{\kappa x_{2}x_{4}}{1-2\kappa}\right)
\left(  -\theta_{1}\theta_{3}+\theta_{1}\theta_{4}-\theta_{3}\theta_{4}\right)
\\
+\frac{\kappa x_{3}x_{4}}{\left(  1-2\kappa\right)  \left(  1+\kappa\right)
}\left\{  \left(  1-\kappa\right)  \theta_{1}\theta_{2}-\left(  1-2\kappa
\right)  \left(  \theta_{1}\theta_{3}-\theta_{2}\theta_{3}\right)
-\kappa\left(  \theta_{1}\theta_{4}-\theta_{2}\theta_{4}\right)  \right\}  ,\\
\zeta_{\alpha,E}=\left[  1-\kappa,2+\kappa,2-2\kappa,1\right]  .
\end{gather*}

\end{example}

We collect formulas for the action of $s_{i}$ on $J_{\alpha,E}$. They will be
expressed in terms of the spectral vector $\zeta_{\alpha,E}=\left[  \alpha
_{i}+1+\kappa c\left(  r_{\alpha}\left(  i\right)  ,E\right)  \right]
_{i=1}^{N}$ and (for $1\leq i<N$)%
\begin{align*}
b_{\alpha,E}\left(  i\right)   &  =\frac{\kappa}{\zeta_{\alpha,E}\left(
i\right)  -\zeta_{\alpha,E}\left(  i+1\right)  }\\
&  =\frac{\kappa}{\alpha_{i}-\alpha_{i+1}+\kappa\left(  c\left(  r_{\alpha
}\left(  i\right)  ,E\right)  -c\left(  r_{\alpha}\left(  i+1\right)
,E\right)  \right)  }.
\end{align*}
The formulas are consequences of the commutation relationships: $s_{j}%
\mathcal{U}_{i}=\mathcal{U}_{i}s_{j}$ for $j<i-1$ and $j>i$; $s_{i}%
\mathcal{U}_{i}s_{i}=\mathcal{U}_{i+1}+\kappa s_{i}$ for $1\leq i<N$. We
examine the action of $s_{i}$ on $J_{\alpha,E}$ with $i<N$. Observe that the
formulas manifest the equation $\left(  s_{i}+b_{\alpha,E}\left(  i\right)
\right)  \left(  s_{i}-b_{\alpha,E}\left(  i\right)  \right)  =1-b_{\alpha
,E}\left(  i\right)  ^{2}$. Suppose that $\alpha_{i}\neq\alpha_{i+1}$,
then\newline(1) $\alpha_{i}<\alpha_{i+1}$ implies%
\begin{align*}
\left(  s_{i}-b_{\alpha,E}\left(  i\right)  \right)  J_{\alpha,E}  &
=J_{s_{i}\alpha,E}\\
\left(  s_{i}+b_{\alpha,E}\left(  i\right)  \right)  J_{s_{i}\alpha,E}  &
=\left(  1-b_{\alpha,E}\left(  i\right)  ^{2}\right)  J_{\alpha,E}.,
\end{align*}
\newline(2) $\alpha_{i}>\alpha_{i+1}$ implies%
\begin{align*}
\left(  s_{i}-b_{\alpha,E}\left(  i\right)  \right)  J_{\alpha,E}  &  =\left(
1-b_{\alpha,E}\left(  i\right)  ^{2}\right)  J_{s_{i}\alpha,E},\\
\left(  s_{i}+b_{\alpha,E}\left(  i\right)  \right)  J_{s_{i}\alpha,E}  &
=J_{\alpha,E},
\end{align*}

Suppose that $\alpha_{i}=\alpha_{i+1}$ then let $j=r_{\alpha}\left(  i\right)
$ (thus $r_{\alpha}\left(  i+1\right)  =j+1$). By definition $b_{\alpha
,E}\left(  i\right)  =\allowbreak\left(  c\left(  j,E\right)  -c\left(
j+1,E\right)  \right)  ^{=1}$. Then\newline(1) $\left\{  j,j+1\right\}
\subset E$ ($b_{\alpha,E}\left(  i\right)  =-1$) implies $s_{i}J_{\alpha
,E}=-J_{\alpha,E}$,\newline(2) $\left\{  j,j+1\right\}  \cap E=\emptyset$
($b_{\alpha,E}\left(  i\right)  =1$) implies $s_{i}J_{\alpha,E}=-J_{\alpha,E}%
$,\newline(3) $\left(  j,j+1\right)  \in E^{C}\times E$ and $E\in
\mathcal{E}_{0}$, or $\left(  j,j+1\right)  \in E\times E^{C}$ and
$E\in\mathcal{E}_{1}$ implies%
\begin{align*}
\left(  s_{i}-b_{\alpha,E}\left(  i\right)  \right)  J_{\alpha,E}  &
=J_{\alpha,s_{j}E}\text{,}\\
\left(  s_{i}+b_{\alpha,E}\left(  i\right)  \right)  J_{\alpha,s_{j}E}  &
=\left(  1-b_{\alpha,E}\left(  i\right)  ^{2}\right)  J_{\alpha,E},
\end{align*}
\newline(4) $\left(  j,j+1\right)  \in E\times E^{C}$ and $E\in\mathcal{E}%
_{0}$, or $\left(  j,j+1\right)  \in E^{C}\times E$ and $E\in\mathcal{E}_{1}$
implies%
\begin{align*}
\left(  s_{i}-b_{\alpha,E}\left(  i\right)  \right)  J_{\alpha,E}  &  =\left(
1-b_{\alpha,E}\left(  i\right)  ^{2}\right)  J_{\alpha,s_{j}E},\\
\left(  s_{i}+b_{\alpha,E}\left(  i\right)  \right)  J_{\alpha,s_{j}E}  &
=J_{\alpha,E}.
\end{align*}

The reason for the difference between the $\mathcal{E}_{0}$ and $\mathcal{E}%
_{1}$ cases is that the directions of inductively defining $T_{E}$ are
opposite. One other construction enables the determination of any
$J_{\alpha,E}$ in a finite number of steps starting from $T_{E_{0}}$ for
$E\in\mathcal{E}_{0}$, or from $T_{E_{1}}$ for $E\in\mathcal{E}_{1}$: this
refers to the affine step. Let $w_{N}=s_{1}s_{2}\cdots s_{N-1}$ (a cyclic
shift) and define $\Psi$ acting on $N$-vectors:%
\[
\Psi\left(  a_{1},a_{2},\ldots,a_{N}\right)  =\left(  a_{2},a_{3},\ldots
,a_{N},a_{1}+1\right)  .
\]
Then $J_{\Psi\alpha,E}=x_{N}w_{N}^{-1}J_{\alpha,E}$ and $\left[  \zeta
_{\Psi\alpha,E}\left(  i\right)  \right]  _{i=1}^{N}=\Psi\left[  \zeta
_{\alpha,E}\left(  i\right)  \right]  _{i=1}^{N}$. By a simple calculation we
find $r_{\Psi_{\alpha}}\left(  i\right)  =r_{\alpha}\left(  i+1\right)  $ for
$1\leq i<N$ and $r_{\Psi_{\alpha}}\left(  N\right)  =r_{\alpha}\left(
1\right)  $, that is, $r_{\Psi_{\alpha}}\left(  i\right)  =r_{\alpha}\left(
w_{N}(i)\right)  $ for all $i$. Furthermore%
\[
x_{N}w_{N}^{-1}J_{\alpha,E}=x_{N}J_{\alpha,E}\left(  x_{N},x_{1},x_{2}%
,\ldots,x_{N-1};\theta_{N},\theta_{1},\ldots,\theta_{N-1}\right)  .
\]
For example, if $\alpha=\left(  2,1,4\right)  $ then $x^{\alpha}=x_{1}%
^{2}x_{2}x_{3}^{4}$ and $x^{\Psi\alpha}=x_{1}x_{2}^{4}x_{3}^{3}$ , also
$r_{\alpha}=\left[  2,3,1\right]  $ and $r_{\psi\alpha}=\left[  3,1,2\right]
$.

In \cite{DL2011} the Yang-Baxter graph is used to organize the transformation
formulas in a directed graph. Here is a simplified description: the nodes are
pairs $\left(  \alpha,E\right)  $, and there are arrows from $\left(
\alpha,E\right)  $ to $\left(  \alpha\Psi,E\right)  $, called \textit{affine
steps}, arrows from $\left(  \alpha,E\right)  $ to $\left(  s_{i}%
\alpha,E\right)  $ when $\alpha_{i}<\alpha_{i+1}$, called \textit{steps}, and
arrows from $\left(  \alpha,E\right)  $ to $\left(  \alpha,s_{i}E\right)  $
when $\left(  i,i+1\right)  \in E\times E^{C}$ and $E\in\mathcal{E}_{0}$, or
$\left(  i,i+1\right)  \in E^{C}\times E$ and $E\in\mathcal{E}_{1}$; the
latter are called \textit{jumps }(jumping from one set $E$ to another). The
graphs for $s\mathcal{P}_{m,0},s\mathcal{P}_{m,1}$ have the roots $\left(
\boldsymbol{0},E_{0}\right)  $ and $\left(  \boldsymbol{0},E_{1}\right)  $ respectively.

The graph makes it possible to define a symmetric bilinear form on
$s\mathcal{P}_{m}$, extending the inner product defined above for
$\mathcal{P}_{m},$ having the following properties ($f,g\in s\mathcal{P}_{m}$)%
\begin{align*}
\left\langle wf,wg\right\rangle  &  =\left\langle f,g\right\rangle
,w\in\mathcal{S}_{N},\\
\left\langle x_{i}f,g\right\rangle  &  =\left\langle f,\mathcal{D}%
_{i}g\right\rangle ,~1\leq i\leq N,\\
\deg f  &  \neq\deg g\implies\left\langle f,g\right\rangle =0.
\end{align*}
where $\deg f$ is the bosonic degree ($\deg x^{\alpha}=\left\vert
\alpha\right\vert $). As a consequence the operators $\mathcal{U}_{i}$ are
self-adjoint for this form and $\left(  \alpha,E\right)  \neq\left(
\beta,F\right)  $ implies $\left\langle J_{\alpha,E},J_{\beta,F}\right\rangle
=0$, because of the eigenvector property. There are two products appearing in
the formulas.

\begin{definition}
\label{defPR}For $\lambda\in\mathbb{N}_{0}^{N,+},\alpha\in\mathbb{N}_{0}^{N},$
$E\in\mathcal{E}_{0}\cup\mathcal{E}_{1}$ and $z=0,1$ let%
\begin{align*}
\mathcal{P}\left(  \lambda,E\right)   &  =\prod\limits_{i=1}^{N}\left(
1+\kappa c\left(  i,E\right)  \right)  _{\lambda_{i}}\prod\limits_{1\leq
i<j\leq N}\prod\limits_{l=1}^{\lambda_{i}-\lambda_{j}}\left(  1-\left(
\frac{\kappa}{l+\kappa\left(  c\left(  i,E\right)  -c\left(  j,E\right)
\right)  }\right)  ^{2}\right)  ,\\
\mathcal{R}_{z}\left(  \alpha,E\right)   &  =\prod\limits_{\substack{1\leq
i<j\leq N\\\alpha_{i}<\alpha_{j}}}\left(  1+\frac{\left(  -1\right)
^{z}\kappa}{\alpha_{j}-\alpha_{i}+\kappa\left(  c\left(  r_{\alpha}\left(
j\right)  ,E\right)  -c\left(  r_{\alpha}\left(  i\right)  ,E\right)  \right)
}\right)  ,
\end{align*}
and let $\mathcal{R}\left(  \alpha,E\right)  =\mathcal{R}_{0}\left(
\alpha,E\right)  \mathcal{R}_{1}\left(  \alpha,E\right)  $ ($\mathcal{R}$ is
for \textquotedblleft rearrangement\textquotedblright).
\end{definition}

\begin{theorem}
\label{Jnorm1}Suppose $\lambda\in\mathbb{N}^{N,+}$ and $E\in\mathcal{E}%
_{0}\cup\mathcal{E}_{1}$ then%
\[
\left\Vert J_{\lambda,E}\right\Vert ^{2}=\left\vert T_{E}\right\vert
^{2}\mathcal{P}\left(  \lambda,E\right)  .
\]

\end{theorem}

The norm of $J_{\alpha,E}$ uses the auxiliary product.

\begin{theorem}
\label{Jnorm2}Suppose $\alpha\in\mathbb{N}_{0}^{N},$ $E\in\mathcal{E}_{0}%
\cup\mathcal{E}_{1}$ then
\[
\left\Vert J_{\alpha,E}\right\Vert ^{2}=\mathcal{R}\left(  \alpha,E\right)
^{-1}\left\Vert J_{\alpha^{+},E}\right\Vert ^{2}.
\]

\end{theorem}

These formulas show that the bilinear form is positive-definite for $-\frac
{1}{N}<\kappa<\frac{1}{N}$; the typical term in the product has the form
$\left(  l\geq1\right)  $
\[
\dfrac{\left(  l+\kappa\left(  c\left(  i,E\right)  -c\left(  j,E\right)
+1\right)  \right)  \left(  l+\kappa\left(  c\left(  i,E\right)  -c\left(
j,E\right)  -1\right)  \right)  }{\left(  l+\kappa\left(  c\left(  i,E\right)
-c\left(  j,E\right)  \right)  \right)  ^{2}}%
\]
and $\left\vert c\left(  i,E\right)  -c\left(  j,E\right)  \right\vert \leq
N-1$ (maximum value with the cells $\left\{  \left[  1,N-m\right]  ,\left[
m+1,1\right]  \right\}  $ for $\mathcal{E}_{0}$and with $\left\{  \left[
1,N-m+1\right]  ,\left[  m,1\right]  \right\}  $ for $\mathcal{E}_{1}$; thus
each factor is positive. The formulas are the type-$A$ specialization of
Griffeth's results \cite{G2010}.

\begin{example}
$N=4,m=2,\alpha=\left(  0,1,1,0\right)  $, $E=\{2,3,4\}\in\mathcal{E}_{0},$
$\left[  c\left(  j,E\right)  \right]  _{j=1}^{4}=\left[  1,-2,-1,0\right]  $,
$\zeta_{\alpha,E}=\left[  1-\kappa,2+\kappa,2-2\kappa,1\right]  $
\begin{gather*}
J_{\alpha,E}=\left(  x_{2}x_{3}-\frac{\kappa x_{2}x_{4}}{1-2\kappa}\right)
\left(  -\theta_{1}\theta_{3}+\theta_{1}\theta_{4}-\theta_{3}\theta_{4}\right)
\\
+\frac{\kappa x_{3}x_{4}}{\left(  1-2\kappa\right)  \left(  1+\kappa\right)
}\left\{  \left(  1-\kappa\right)  \theta_{1}\theta_{2}-\left(  1-2\kappa
\right)  \left(  \theta_{1}\theta_{3}-\theta_{2}\theta_{3}\right)
-\kappa\left(  \theta_{1}\theta_{4}-\theta_{2}\theta_{4}\right)  \right\}  ,\\
\left\Vert J_{\alpha,E}\right\Vert ^{2}=\frac{3\left(  1-3\kappa\right)
\left(  1+2\kappa\right)  \left(  1-\kappa\right)  }{\left(  1+\kappa\right)
\left(  1-2\kappa\right)  }.
\end{gather*}

\end{example}

\section{Supersymmetric Polynomials\label{supersym}}

A supersymmetric polynomial of fermionic degree $m$ is a polynomial $p\in
s\mathcal{P}_{m}$ which satisfies $wp=p$ for all $w\in\mathcal{S}_{N};$ the
minimal equivalent condition is $p\left(  xs_{i};\theta s_{i}\right)
=p\left(  x;\theta\right)  $ for $1\leq i<N$. In this section we consider such
polynomials which arise as $\sum\limits_{w\in\mathcal{S}_{N}}wJ_{\alpha,E}$
for some fixed $\alpha,E$. These polynomials are simultaneous eigenfunctions
of the commutative set $\left\{  \sum_{i=1}^{N}\mathcal{U}_{i}^{k}:1\leq k\leq
N\right\}  $. The tableaux $Y_{E}$ are useful for labeling $\mathcal{S}_{N}%
$-orbits of $J_{\alpha,E}$

\begin{definition}
\label{defMaE}For $\alpha\in\mathbb{N}_{0}^{N}$ and $E\in\mathcal{E}_{0}%
\cup\mathcal{E}_{1}$ let $\left\lfloor \alpha,E\right\rfloor $ denote the
tableau obtained from $Y_{E}$ by replacing $i$ by $\alpha_{i}^{+}$ for $1\leq
i\leq N$. Let $\mathcal{M}\left(  \alpha,E\right)  \allowbreak=\mathrm{span}%
\left\{  J_{\beta,F}:\left\lfloor \beta,F\right\rfloor =\left\lfloor
\alpha,E\right\rfloor \right\}  $ .
\end{definition}

Example: let $N=8,m=3,E=\left\{  2,5,7,8\right\}  ,\alpha=\left(
3,5,6,2,2,1,4,4\right)  ,\alpha^{+}=\left(  6,5,4,4,3,2,2,1\right)  $ and%
\[
\left\lfloor \alpha,E\right\rfloor =%
\begin{bmatrix}
1 & 2 & 4 & 4 & 6\\
\circ & 2 & 3 & 5 &
\end{bmatrix}
.
\]

\begin{theorem}
(\cite[Prop. 5.2]{DL2011})Suppose $\alpha\in\mathbb{N}_{0}^{N}$ and
$E\in\mathcal{E}_{0}\cup\mathcal{E}_{1}$, then there is a series of
transformations of the form $as_{i}+b$ mapping $J_{\alpha,E}$ to $J_{\beta,F}$
if and only if $\left\lfloor \beta,F\right\rfloor =\left\lfloor \alpha
,E\right\rfloor $.
\end{theorem}

It is a consequence of the transformation rules that if $\left\lfloor
\beta,F\right\rfloor =\left\lfloor \alpha,E\right\rfloor $ then the spectral
vector $\zeta_{\beta,F}$ is a permutation of $\zeta_{\alpha,E}$. Furthermore
$\mathcal{M}\left(  \alpha,E\right)  $ is an $\mathcal{S}_{N}$-module.

\begin{theorem}
(\cite[Thm. 5.9]{DL2011}) Suppose $\alpha\in\mathbb{N}_{0}^{N}$ and
$E\in\mathcal{E}_{0}\cup\mathcal{E}_{1}$ and $\left\lfloor \alpha
,E\right\rfloor $ is column-strict (the entries in column 1 are strictly
decreasing) then there is a unique symmetric polynomial (up to constant
multiplication) in $\mathcal{M}\left(  \alpha,E\right)  $ otherwise there is
no nonzero $\mathcal{S}_{N}$-invariant.
\end{theorem}

The paper \cite{DLM2003a} defined a superpartition with $N$ parts and
fermionic degree $m$ as an $N$-tuple $\left(  \Lambda_{1},\ldots,\Lambda
_{m};\Lambda_{m+1},\ldots,\Lambda_{N}\right)  $ which satisfies $\Lambda
_{1}>\Lambda_{2}>\cdots>\Lambda_{m}$ and $\Lambda_{m+1}\geq\Lambda_{m+2}%
\geq\cdots\geq\Lambda_{N}$. Suppose $\alpha\in\mathbb{N}_{0}^{N}$ and
$E\in\mathcal{E}_{0}$ and $\left\lfloor \alpha,E\right\rfloor $ is column
strict, then $\Lambda_{i}=\left\lfloor \alpha,E\right\rfloor \left[
m+2-i,1\right]  $ for $1\leq i\leq m$ and $\Lambda_{i}=\left\lfloor
\alpha,E\right\rfloor \left[  1,N+1-i\right]  $ for $m+1\leq i\leq N$, and
also $\Lambda_{m}>\Lambda_{N}.$ Suppose $\alpha\in\mathbb{N}_{0}^{N}$ and
$E\in\mathcal{E}_{1}$ and $\left\lfloor \alpha,E\right\rfloor $ is column
strict, then $\Lambda_{i}=\left\lfloor \alpha,E\right\rfloor \left[
m+1-i,1\right]  $ for $1\leq i\leq m$ and $\Lambda_{i}=\left\lfloor
\alpha,E\right\rfloor \left[  1,N+2-i\right]  $ for $m+1\leq i\leq N$, and
also $\Lambda_{m}\leq\Lambda_{N}$ (because $\Lambda_{m}=\left\lfloor
\alpha,E\right\rfloor \left[  1,1\right]  $ and $\Lambda_{N}=\left\lfloor
\alpha,E\right\rfloor \left[  1,2\right]  $. Thus the inequalities
$\Lambda_{m}>\Lambda_{N}$ and $\Lambda_{m}\leq\Lambda_{N}$ distinguish
$\mathcal{E}_{0}$ from $\mathcal{E}_{1}.$

Suppose $\left\lfloor \alpha,E\right\rfloor $ is column-strict. We will
determine formulas for the supersymmetric polynomial in $\mathcal{M}\left(
\alpha,E\right)  $ and its norm. This material is the specialization of
results in \cite{DL2011}. The idea is based on a technique of Baker and
Forrester \cite{BF1999}. Some auxiliary concepts are required. Consider the
set of $F$ such that $\left\lfloor \alpha,F\right\rfloor =\left\lfloor
\alpha,E\right\rfloor $. Among these there is one minimizing $\mathrm{inv}%
\left(  F\right)  $ and one maximizing $\mathrm{inv}\left(  F\right)  $.

\begin{definition}
Suppose $E\in\mathcal{E}_{0}$ then the root $E_{R}$ and the sink $E_{S}$
satisfy%
\begin{align*}
\mathrm{inv}\left(  E_{R}\right)   &  =\min\left\{  \mathrm{inv}\left(
F\right)  :\left\lfloor \alpha,F\right\rfloor =\left\lfloor \alpha
,E\right\rfloor \right\}  ,\\
\mathrm{inv}\left(  E_{S}\right)   &  =\max\left\{  \mathrm{inv}\left(
F\right)  :\left\lfloor \alpha,F\right\rfloor =\left\lfloor \alpha
,E\right\rfloor \right\}  .
\end{align*}
The root and the sink are produced by minimizing the entries of $F$ in row 1,
respectively minimizing the entries of $F$ in column 1 . For $E\in
\mathcal{E}_{1}$ the definitions of $E_{R}$ and $E_{S}$ are reversed.
\end{definition}

Example: Let $N=10,m=3,E=\left\{  1,4,7,10\right\}  $, $\alpha=\left(
3322221100\right)  $ so that%
\[
Y_{E}=%
\begin{bmatrix}
10 & 9 & 8 & 6 & 5 & 3 & 2\\
\circ & 7 & 4 & 1 &  &  &
\end{bmatrix}
,\left\lfloor \alpha,E\right\rfloor =%
\begin{bmatrix}
0 & 0 & 1 & 2 & 2 & 2 & 3\\
\circ & 1 & 2 & 3 &  &  &
\end{bmatrix}
;
\]
there are $16$ sets $F$ with $\left\lfloor \alpha,F\right\rfloor =\left\lfloor
a,E\right\rfloor $, and $E_{R}=\left\{  2,6,8,10\right\}  ,E_{S}=\left\{
1,3,7,10\right\}  $.

There is an analog of the product $\mathcal{R}_{z}\left(  \alpha,E\right)  $
for content vectors:

\begin{definition}
For $E\in\mathcal{E}_{0},F\in\mathcal{E}_{1}$ and $z=0,1$ let%
\begin{align*}
\mathcal{C}_{z}^{\left(  0\right)  }\left(  E\right)   &  =\prod
\limits_{\substack{1\leq i<j<N\\\left(  i,j\right)  \in E\times E^{C}}}\left(
1+\frac{\left(  -1\right)  ^{z}}{c\left(  i,E\right)  -c\left(  j,E\right)
}\right) \\
\mathcal{C}_{z}^{\left(  1\right)  }\left(  F\right)   &  =\prod
\limits_{\substack{1\leq i<j<N\\\left(  i,j\right)  \in F^{C}\times F}}\left(
1+\frac{\left(  -1\right)  ^{z}}{c\left(  i,F\right)  -c\left(  j,F\right)
}\right)  .
\end{align*}

\end{definition}

\begin{theorem}
\label{superPthm}Suppose $\lambda\in\mathbb{N}_{0}^{N,+},E\in\mathcal{E}%
_{k},k=0$ or $1$, and $\left\lfloor \lambda,E\right\rfloor $ is column-strict
then
\[
p_{\lambda,E}=\sum_{\left\lfloor \alpha,F\right\rfloor \in\mathcal{T}\left(
\lambda,E\right)  }\frac{\mathcal{R}_{1}\left(  \alpha,F\right)  }%
{\mathcal{C}_{1}^{\left(  k\right)  }\left(  F\right)  }J_{\alpha,F}%
\]
is the supersymmetric polynomial in $\mathcal{M}\left(  \alpha,E\right)  $,
unique when the coefficient of $x^{\lambda}$ is $\sum\limits_{\left\lfloor
\lambda,F\right\rfloor \in\mathcal{T}\left(  \lambda,E\right)  }\dfrac
{1}{\mathcal{C}_{1}^{\left(  k\right)  }\left(  F\right)  }T_{F}$.
\end{theorem}

\begin{proof}
We show that $s_{i}p_{\lambda,E}=p_{\lambda,E}$ for each $i<N$. The space
$\mathcal{M}\left(  \alpha,E\right)  $ decomposes as a sum of one and two
dimensional spaces invariant under the action of $s_{i}$. This leads to
adjacency relations for the coefficients in $\sum_{\left(  \alpha,F\right)
}A\left(  \alpha,F\right)  J_{\alpha,F}$. There are 3 cases (1) $\alpha
_{i}<\alpha_{i+1}$ (2) $\alpha_{i}=\alpha_{i+1}$, $\left\vert c\left(
r_{\alpha}\left(  i\right)  ,F\right)  -c\left(  r_{\alpha}\left(  i+1\right)
,F\right)  \right\vert \geq2$, (3) $\alpha_{i}=\alpha_{i+1}$, $c\left(
r_{\alpha}\left(  i\right)  ,F\right)  -c\left(  r_{\alpha}\left(  i+1\right)
,F\right)  =1$. The case $\alpha_{i}=\alpha_{i+1}$, $c\left(  r_{\alpha
}\left(  i\right)  ,F\right)  -c\left(  r_{\alpha}\left(  i+1\right)
,F\right)  =-1$ does not occur because $\left\lfloor \lambda,E\right\rfloor $
is column-strict. In case (3) $s_{i}J_{\alpha,F}=J_{\alpha,F}$. Now suppose
$\alpha_{i}<\alpha_{i+1}$ and $b_{\alpha,F}\left(  i\right)  =\dfrac{\kappa
}{\xi_{i}\left(  \alpha,F\right)  -\xi_{i+1}\left(  \alpha,F\right)  }$, so
that $\dfrac{\kappa}{\xi_{i}\left(  s_{i}\alpha,F\right)  -\xi_{i+1}\left(
s_{i}\alpha,F\right)  }=-b_{\alpha,F}\left(  i\right)  $. The transformation
formula $\left(  s_{i}-b_{\alpha,E}\left(  i\right)  \right)  J_{\alpha
,F}=J_{s_{i}\alpha,F}$ implies $A\left(  \alpha,F\right)  =\left(
1+b_{\alpha,F}\left(  i\right)  \right)  A\left(  s_{i}\alpha,F\right)  $.
Combine this with%
\[
\frac{\mathcal{R}_{z}\left(  \alpha,F\right)  }{\mathcal{R}_{z}\left(
s_{i}\alpha,F\right)  }=1-\left(  -1\right)  ^{z}b_{\alpha,F}\left(  i\right)
\]
(which follows from the product for $a$ having the extra pair $\left(
i,i+1\right)  $ with $\alpha_{i}<\alpha_{i+1}$) leads to%
\[
\frac{A\left(  s_{i}\alpha,F\right)  }{\mathcal{R}_{1}\left(  s_{i}%
\alpha,F\right)  }=\frac{A\left(  \alpha,F\right)  }{\mathcal{R}_{1}\left(
\alpha,F\right)  }=A\left(  \lambda,E\right)  ,
\]
since there is a series of steps transforming $\alpha$ to $\lambda=\alpha^{+}$.

Suppose $\alpha_{i}=\alpha_{i+1}$ and $j=r_{\alpha}\left(  i\right)  ,$then
let $b_{F}\left(  j\right)  =\left(  c\left(  j,F\right)  -c\left(
j+1,F\right)  \right)  ^{-1}$. If $E\in\mathcal{E}_{0}$ suppose $\left(
j,j+1\right)  \in F^{C}\times\left(  F\backslash\left\{  N\right\}  \right)  $
or if $E\in\mathcal{E}_{1}$ then suppose $\left(  j,j+1\right)  \in
F\times\left(  F^{C}\backslash\left\{  N\right\}  \right)  $; in both cases
$\left(  s_{i}-b_{\alpha,F}\left(  j\right)  \right)  J_{\alpha,F}%
=J_{\alpha,s_{j}F}$. This implies $A\left(  \alpha,F\right)  =\left(
1+b_{F}\left(  j\right)  \right)  A\left(  \alpha,s_{j}F\right)  $. Combine
with ($k=0,1$)
\[
\frac{\mathcal{C}_{z}^{\left(  k\right)  }\left(  s_{j}E\right)  }%
{\mathcal{C}_{z}^{\left(  k\right)  }\left(  E\right)  }=1-\left(  -1\right)
^{z}b_{E}\left(  j\right)
\]
to obtain%
\[
\frac{A\left(  \alpha,F\right)  }{A\left(  \alpha,s_{j}F\right)  }%
=1+b_{E}\left(  j\right)  =\frac{\mathcal{C}_{1}^{\left(  k\right)  }\left(
s_{j}F\right)  }{\mathcal{C}_{1}^{\left(  k\right)  }\left(  F\right)  },
\]
for $E\in\mathcal{E}_{k}$. So every coefficient in $\sum_{\left(
\alpha,F\right)  }A\left(  \alpha,F\right)  J_{\alpha,F}$ can be expressed in
terms of $A\left(  \lambda,E\right)  ,$ which is now the unique undetermined constant.
\end{proof}

Besides the supersymmetry property the polynomial $p_{\lambda,E}$ satisfies
\[
\sum_{j=1}^{N}\mathcal{U}_{j}^{s}p_{\lambda,E}=\sum_{j=1}^{N}\left(
\lambda_{j}+1+\kappa c\left(  i,E_{R}\right)  \right)  ^{s}p_{\lambda,E}%
\]
for $s\geq1$; this relation holding for $1\leq s\leq N$ together with
supersymmetry defines $p$ (up to a multiplicative constant).

Next we describe the process for computing the norm of $p_{\lambda,E}.$ Recall
that $\left\Vert J_{\lambda,E}\right\Vert ^{2}=\left\vert T_{E}\right\vert
^{2}\mathcal{P}\left(  \lambda,E\right)  $ and $\left\Vert J_{\alpha
,E}\right\Vert ^{2}=\left\vert T_{E}\right\vert ^{2}\mathcal{P}\left(
\alpha^{+},E\right)  \left(  \mathcal{R}_{0}\left(  \alpha,E\right)
\mathcal{R}_{1}\left(  \alpha,E\right)  \right)  ^{-1}$ (Theorem
\ref{Jnorm2}).. If $\left\lfloor \lambda,F\right\rfloor =\left\lfloor
\lambda,E\right\rfloor $ then $\mathcal{P}\left(  \lambda,F\right)
=\mathcal{P}\left(  \lambda,E\right)  $ (some of the factors in the product
may be permuted). To determine $\left\Vert p\right\Vert ^{2}$ fix some
$\left(  \alpha,F\right)  \in\left\lfloor \lambda,E_{R}\right\rfloor $ then
$\sum\limits_{w\in\mathcal{S}_{N}}wJ_{\alpha,F}=cp_{\lambda,E}$ for some
constant $c$ (because there is only one invariant in $\mathcal{M}\left(
\alpha,E\right)  $). For now suppose $E\in\mathcal{E}_{0}$. Then%
\begin{align}
c\left\Vert p_{\lambda,E}\right\Vert ^{2}  &  =\sum\limits_{w\in
\mathcal{S}_{N}}\left\langle wJ_{\alpha,F},p_{\lambda,E}\right\rangle
=N!\left\langle J_{\alpha,F},p_{\lambda,E}\right\rangle \label{cpsq}\\
&  =N!\frac{\mathcal{R}_{1}\left(  \alpha,F\right)  }{\mathcal{C}_{1}^{\left(
0\right)  }\left(  F\right)  }\left\Vert J_{\alpha,F}\right\Vert ^{2}%
=N!\frac{\mathcal{R}_{1}\left(  \alpha,F\right)  }{\mathcal{C}_{1}^{\left(
0\right)  }\left(  F\right)  }\frac{\left\vert F\right\vert ^{2}%
\mathcal{P}\left(  \lambda,F\right)  }{\mathcal{R}_{1}\left(  \alpha,F\right)
\mathcal{R}_{0}\left(  \alpha,F\right)  }\nonumber\\
&  =N!\mathcal{P}\left(  \lambda,E_{R}\right)  \left(  m+1\right)
\frac{\mathcal{C}_{0}^{\left(  0\right)  }\left(  F\right)  }{\mathcal{R}%
_{0}\left(  \alpha,F\right)  }.\nonumber
\end{align}
The constant $c$ can be determined for $\alpha=\lambda^{-}$ (the
non-decreasing rearrangement of $\lambda$) and $F=E_{R}$. Let $w_{0}%
=r_{\lambda^{-}}$ so that $w_{0}\lambda_{-}=\lambda.$ By the triangular
property of $\vartriangleright$ and the minimal property of $\lambda^{-}$ it
follows that $J_{\lambda^{-},E_{R}}=x^{\lambda^{-}}w_{0}^{-1}T_{E_{R}}%
+\sum_{\beta\vartriangleleft\lambda^{-}}x^{\beta}v_{\lambda^{-},\beta,E_{R}%
}\left(  \kappa;\theta\right)  $ (see formula (\ref{dNJP})) where $\beta
^{+}\neq\lambda$ ($\beta^{+}\prec\lambda$). Thus $x^{\lambda}$ appears in
$\sum\limits_{w\in\mathcal{S}_{N}}wJ_{\lambda^{-},E_{R}}$ exactly when
$w=w_{1}w_{0}$ and $w_{1}\in G_{\lambda}$, the stabilizer of $\lambda$, and so
$\sum\limits_{w\in G_{\lambda}}wT_{E_{R}}$ becomes the relevant quantity.

\begin{lemma}
\label{coefGE}The coefficient of $T_{E_{S}}$ in $\sum\limits_{w\in G_{\lambda
}}wT_{E_{R}}$ is $\#G_{E_{R}}$ , the order of the stabilizer of $T_{E_{R}}$.
\end{lemma}

\begin{proof}
The stabilizer of $\lambda$ is a direct product of symmetric groups of
intervals ($\mathcal{S}\left[  a,b\right]  =\left\{  w\in\mathcal{S}%
_{N}:w\left(  i\right)  =i,i<a,i>b\right\}  $), thus if $\lambda_{a-1}%
>\lambda_{a}=\lambda_{a+1}=\cdots=\lambda_{b}>\lambda_{b+1}$ then
$\mathcal{S}\left[  a,b\right]  $ is a factor of $G_{\lambda}$. There are two
possibilities for such an interval (1) $\left\{  a,a+1,\ldots,b\right\}  $ are
all in row 1 of $Y_{E_{R}}$ then $w\in\mathcal{S}\left[  a,b\right]  $ implies
$wT_{E_{R}}=T_{E_{R}}$ (and $wJ_{\lambda,E_{R}}=J_{\lambda,E_{R}})$ (2) one of
$a,a+1,\ldots,b$ is in column 1 (excluding $\left[  1,1\right]  $) of
$Y_{E_{R}}$. In case (2) $b$ is in column 1 for $E_{R}$ and $a$ is in column 1
for $E_{S}$ (reversed for $E\in\mathcal{E}_{1}$).

For simplification of the argument assume there is only one interval in case
(2): then by the transformation laws (Theorem \ref{TEexist})
\[
s_{a}s_{a+1}\cdots s_{b-1}T_{E_{R}}=T_{E_{S}}+\sum\left\{  b_{E}%
T_{E}:\mathrm{inv}\left(  E_{R}\right)  \leq\mathrm{inv}\left(  E\right)
<\mathrm{inv}\left(  E_{R}\right)  \right\}
\]
and $w\left(  b\right)  =a$ implies $w=s_{a}s_{a+1}\cdots s_{b-1}w_{1}$ where
$w_{1}\left(  b\right)  =b$ , that is, there are $\left(  b-a\right)  !$
permutations taking $b$ to $a$. In general the coefficient of $T_{E_{S}}$ in
$\sum\limits_{w\in G_{\lambda}}wT_{E_{R}}$ is a product to which the intervals
in case (1) or case (2)  contribute $\left(  b_{k}-a_{k}+1\right)  !,\left(
b_{k}-a_{k}\right)  !$ respectively. The product is exactly $\#G_{E_{R}}$ .
\end{proof}

\begin{theorem}
Suppose $p_{\lambda,E}$ is as defined in Theorem \ref{superPthm} and $E_{R}%
\in\mathcal{E}_{k}$ with $k=0,1$ then%
\[
\left\Vert p_{\lambda,E}\right\Vert ^{2}=\nu_{k}\frac{N!}{\#G_{E_{R}}}%
\frac{\mathcal{C}_{0}^{\left(  k\right)  }\left(  E_{R}\right)  \mathcal{P}%
\left(  \lambda,E_{R}\right)  }{\mathcal{R}_{0}\left(  \lambda^{-}%
,E_{R}\right)  \mathcal{C}_{1}^{\left(  k\right)  }\left(  E_{S}\right)  },
\]
with $\nu_{0}=m+1,\nu_{1}=N-m+1.$
\end{theorem}

\begin{proof}
From formula (\ref{cpsq}) we find%
\begin{align*}
cp_{\lambda,E}  &  =\sum_{w\in\mathcal{S}_{N}}wJ_{\lambda^{-},E_{R}}%
=c\sum_{\left\lfloor \alpha,F\right\rfloor \in\mathcal{T}\left(
\lambda,E\right)  }\frac{\mathcal{R}_{1}\left(  \alpha,F\right)  }%
{\mathcal{C}_{1}^{\left(  k\right)  }\left(  F\right)  }J_{\alpha,F}\\
c\left\Vert p_{\lambda,E}\right\Vert ^{2}  &  =\nu_{k}N!\mathcal{P}\left(
\lambda,E_{R}\right)  \frac{\mathcal{C}_{0}^{\left(  k\right)  }\left(
E_{R}\right)  }{\mathcal{R}_{0}\left(  \lambda^{-},E_{R}\right)  }.
\end{align*}
In the first line the equation for the coefficient of $x^{\lambda}T_{E_{S}}$
is $\#G_{E_{R}}=\dfrac{c}{\mathcal{C}_{1}^{\left(  k\right)  }\left(
E_{S}\right)  }$ (By Lemma \ref{coefGE}). Dividing the second line by
$c=\left(  \#G_{E_{R}}\right)  \mathcal{C}_{1}^{\left(  k\right)  }\left(
E_{S}\right)  $ leads to the stated formula for $\left\Vert p_{\lambda
,E}\right\Vert ^{2}$.
\end{proof}

\begin{example}
$N=4,$ $\lambda=(2,1,1,0),E=\left\{  2,3,4\right\}  \in\mathcal{E}_{0}$
(isotype $\left(  2,1,1\right)  $)%
\begin{align*}
\left\lfloor \lambda,E\right\rfloor  &  =%
\begin{bmatrix}
0 & 1 & \\
\circ & 1 & 2
\end{bmatrix}
,\\
p_{\lambda,E}  &  =\theta_{1}\theta_{2}\left(  x_{1}-x_{2}\right)  \left\{
\left(  x_{1}x_{2}+x_{3}x_{4}\right)  \left(  x_{3}+x_{4}\right)  -2x_{3}%
x_{4}\left(  x_{1}+x_{2}\right)  \right\}  +\ldots\\
&  =\left(  \theta_{1}\theta_{2}+\theta_{1}\theta_{3}-2\theta_{1}\theta
_{4}+\theta_{2}\theta_{4}+\theta_{3}\theta_{4}\right)  x_{1}^{2}x_{2}%
x_{3}+\ldots
\end{align*}
and%
\begin{align*}
\left\Vert p_{\lambda,E}\right\Vert ^{2}  &  =\ast\left(  1-2\kappa\right)
\left(  2-3\kappa\right)  \left(  1-4\kappa\right) \\
\sum_{i=1}^{4}\mathcal{U}_{i}^{2}p_{\lambda,E}  &  =6\left(  \kappa
^{2}-2\kappa+3\right)  p_{\lambda,E}.
\end{align*}

\end{example}

There is a simplification for the ratio $\dfrac{\mathcal{P}\left(
\lambda,E_{R}\right)  }{\mathcal{R}_{0}\left(  \lambda^{-},E_{R}\right)  }$ by
use of
\[
\mathcal{R}_{0}\left(  \lambda^{-},E_{R}\right)  =\prod
\limits_{\substack{1\leq\ell<k\leq N\\\lambda_{\ell}>\lambda_{k}}}\left(
1+\frac{\kappa}{\lambda_{\ell}-\lambda_{k}+\kappa\left(  c\left(  \ell
,E_{R}\right)  -c\left(  k,E_{R}\right)  \right)  }\right)  ,
\]
which follows from Definition \ref{defPR} by setting $i=w^{-1}\left(
k\right)  ,j=w^{-1}\left(  \ell\right)  $ where $w=r_{\lambda^{-}}$; then
$\lambda_{j}^{-}=\lambda_{\ell},\lambda_{i}^{-}=\lambda_{k}$. It is possible
that $i<j$ does not imply $k>\ell$ (if $\lambda_{i}=\lambda_{j}$) but
$\lambda_{\ell}>\lambda_{k}$ does imply $\ell<k$ and $i<j$. In the product
replace $\ell$ by $i$ and $k$ by $j$, and cancel out part of the factor in
$\mathcal{P}\left(  \lambda,E_{R}\right)  $ for $l=\lambda_{i}-\lambda_{j}$ .
Introduce a utility product to allow a more concise formula:%
\[
\Pi_{0}\left(  n,d\right)  :=\left(  1-\frac{\kappa}{n+d\kappa}\right)
\prod\limits_{l=1}^{n-1}\left(  1-\left(  \frac{\kappa}{l+d\kappa}\right)
^{2}\right)  ,
\]
then%
\[
\dfrac{\mathcal{P}\left(  \lambda,E_{R}\right)  }{\mathcal{R}_{0}\left(
\lambda^{-},E_{R}\right)  }=\prod\limits_{i=1}^{N}\left(  1+\kappa c\left(
i,E_{R}\right)  \right)  _{\lambda_{i}}\prod\limits_{\substack{1\leq i<j\leq
N\\\lambda_{i}>\lambda_{j}}}\Pi_{0}\left(  \lambda_{i}-\lambda_{j},c\left(
i,E_{R}\right)  -c\left(  j,E_{R}\right)  \right)  .
\]
As Griffeth \cite{G2010} pointed out this is a complicated product that
conceals possible cancellations of numerator and denominator factors, because
the content vector need not have consecutive entries differing by $1,$as
happens in the scalar case where $\left[  c\left(  i,T\right)  \right]
_{i=1}^{N}=\left[  N-1,N-2,\ldots,1,0\right]  $. There are some special cases
of low degree where this ratio reduces to a polynomial in $\kappa$; to be
discussed in the next section.

\section{Poincar\'{e} Series and Generators\label{Pseries}}

This section concerns combinatorial aspects and generating supersymmetric
polynomials by multiplying basis elements by symmetric polynomials. We showed
that the supersymmetric polynomials in $s\mathcal{P}_{m}$ are labeled by
column-strict tableaux of shape $\left(  N-m,1^{m}\right)  $ and $\left(
N-m+1,1^{m-1}\right)  .$Stanley \cite[p.379]{S1999} found the formula for the
number $t_{n}$ of tableaux of a given shape $\tau$, with entries
non-decreasing in each row and each column, and with the sum of the entries
equalling $n$%
\[
\sum_{n=0}^{\infty}t_{n}q^{n}=H_{\tau}\left(  q\right)  =\prod\limits_{\left(
i,j\right)  \in\tau}\left(  1-q^{h\left(  i,j;\tau\right)  }\right)  ^{-1},
\]
where the product is over the boxes of the Ferrers diagram of $\tau$ and
$h\left(  i,j;\tau\right)  $ is the length of the hook at $\left(  i,j\right)
$. To modify this series for column-strict tableaux add $j-1$ to each entry of
row $j$ for $1\leq j\leq\ell\left(  \tau\right)  $, thus adding $n_{\tau
}:=\sum_{i=1}^{\ell\left(  \tau\right)  }\left(  j-1\right)  \tau_{j}$ to $n$.
Recall the shifted $q$-factorial: $\left(  a;q\right)  _{n}=\prod
\limits_{i=1}^{n}\left(  1-aq^{i-1}\right)  $. For $\tau=\left(
N-m,1^{m}\right)  $ it is easy to find $H_{\tau}\left(  q\right)  =\left\{
\left(  1-q^{N}\right)  \left(  q;q\right)  _{m}\left(  q;q\right)
_{N-m-1}\right\}  ^{-1}$ and $n_{\tau}=\frac{m\left(  m+1\right)  }{2}$.
Replace $m$ by $m-1$ to obtain these quantities for $\left(  N-m+1,1^{m-1}%
\right)  $. Thus the number of supersymmetric polynomials of isotype $\left(
N-m,1^{m}\right)  $ and bosonic degree $n$ is the coefficient of $q^{n}$ in%
\[
q^{m\left(  m+1\right)  /2}\left\{  \left(  1-q^{N}\right)  \left(
q;q\right)  _{m}\left(  q;q\right)  _{N-m-1}\right\}  ^{-1}%
\]
For example let $\tau=\left(  2,1,1\right)  $ then the series begins
$q^{3}+2q^{4}+4q^{5}+6q^{6}+10q^{7}+\ldots$.

Suppose there are supersymmetric polynomials $p_{1},p_{2},\cdots,p_{k}%
\in\mathcal{P}_{m,0}$ of bosonic degrees $n_{1},n_{2},\ldots,n_{k}$ such that
every supersymmetric polynomial $f\in s\mathcal{P}_{m,0}$ has a unique
expansion $f\left(  x;\theta\right)  =\sum_{i=1}^{k}h_{i}\left(  x\right)
p_{i}\left(  x;\theta\right)  $ where each $h_{i}$ is symmetric ($h_{i}\left(
xw\right)  =h_{i}\left(  x\right)  $ for all $w\in\mathcal{S}_{N}$) . Define
the generating function $Q_{N,m}\left(  q\right)  =\sum_{i=1}^{k}q^{n_{i}}$
then%
\[
Q_{N,m}\left(  q\right)  \prod\limits_{i=1}^{N}\left(  1-q^{i}\right)
^{-1}=q^{m\left(  m+1\right)  /2}\left\{  \left(  1-q^{N}\right)  \left(
q;q\right)  _{m}\left(  q;q\right)  _{N-m-1}\right\}  ^{-1},
\]
because $\left(  q;q\right)  _{N}^{-1}$ is the Poincar\'{e} series for
symmetric polynomials in $N$ variables, graded by degree. Thus%
\begin{equation}
Q_{N,m}\left(  q\right)  =q^{m\left(  m+1\right)  /2}\frac{\left(  q;q\right)
_{N-1}}{\left(  q;q\right)  _{m}\left(  q;q\right)  _{N-m-1}}=q^{m\left(
m+1\right)  /2}%
\genfrac{[}{]}{0pt}{}{N-1}{m}%
_{q}. \label{QNmseries}%
\end{equation}
The Gaussian coefficient is defined by $%
\genfrac{[}{]}{0pt}{}{a}{b}%
_{q}=\frac{\left(  q;q\right)  _{a}}{\left(  q;q\right)  _{b}\left(
q;q\right)  _{a-b}}$. Considering the decomposition $s\mathcal{P}%
_{m}=s\mathcal{P}_{m,0}\oplus s\mathcal{P}_{m,1}$ we find%
\[
Q_{N,m}\left(  q\right)  +Q_{N,m-1}\left(  q\right)  =q^{m\left(  m-1\right)
/2}%
\genfrac{[}{]}{0pt}{}{N}{m}%
_{q}.
\]
Since $\lim\limits_{q\rightarrow1}%
\genfrac{[}{]}{0pt}{}{a}{b}%
_{q}=\binom{a}{b}$ we note that $Q_{N,m}\left(  1\right)  +Q_{N,m-1}\left(
1\right)  =\binom{N}{m}=\dim\mathcal{P}_{m}$; this is a consequence of a
reciprocity theorem for representations. With a similar calculation we find
the Poincar\'{e} series for the supersymmetric polynomials in $s\mathcal{P}%
_{m,1}$ to be $q^{m\left(  m-1\right)  /2}\left\{  \left(  q;q\right)
_{m}\left(  q;q\right)  _{N-m}\right\}  ^{-1}$.

This information leads to determining the set $\left\{  p_{j}\left(
x;\theta\right)  \right\}  $ of supersymmetric polynomials which are
eigenfunctions of $\sum_{i=1}^{N}\mathcal{U}_{i}^{s}$ for $s\geq1$ and which
are generators as described above. The idea is to involve two other meanings
of the Gaussian coefficient. Fix two integers $k,\ell\geq1$ and consider two
(well-known) counting problems:

\begin{enumerate}
\item $A\left(  k,\ell,n\right)  =\#\left\{  \left(  j_{1},\ldots,j_{\ell
}\right)  :0\leq j_{1}\leq j_{2}\leq\cdots\leq j_{\ell}\leq k,\sum_{i=1}%
^{\ell}j_{i}=n\right\}  $ (the number of partitions of $n$ with length
$\leq\ell$), and let $a\left(  k,\ell\right)  =\sum_{n=0}^{k\ell}A\left(
k,\ell,n\right)  q^{n}$;

\item $B\left(  k,\ell,n\right)  =\#\left\{  F\subset\left\{  1,2,\ldots
,k+\ell\right\}  :\#F=\ell,\mathrm{inv}\left(  E\right)  =n\right\}  $, and
$b\left(  k,\ell\right)  =\sum_{n=0}^{k\ell}B\left(  k,\ell,n\right)  q^{n}$.
\end{enumerate}

Then $a\left(  k,\ell\right)  =%
\genfrac{[}{]}{0pt}{}{k+\ell}{k}%
_{q}=b\left(  k,\ell\right)  $. We sketch a proof for the first equation and
set up a bijection to prove the second. Break up the sum for $a\left(
k,\ell\right)  $ by summing over the values of $j_{\ell}$: \thinspace$A\left(
k,\ell,n\right)  =\sum_{s=0}^{k}A\left(  s,\ell-1,n-s\right)  $; this implies
$a\left(  k,\ell\right)  =\sum_{s=0}^{k}q^{s}a\left(  s,\ell-1\right)  $ and
$a\left(  k,\ell\right)  -a\left(  k-1,\ell\right)  =q^{k}a\left(
k,\ell-1\right)  $. Induction on $m+k$ together with $a\left(  k,0\right)
=1=a\left(  0,\ell\right)  $ proves that $a\left(  k,\ell\right)  =%
\genfrac{[}{]}{0pt}{}{k+\ell}{k}%
_{q}$.

For the bijection let $F=\left\{  i_{1},i_{2},\ldots,i_{\ell}\right\}  $ with
$i_{1}<\cdots<i_{l}$ and $\mathrm{inv}\left(  E\right)  =n$, then for $1\leq
u\leq\ell$ define $j_{u}=\#\left\{  v\in F^{C}:v>i_{\ell+1-u}\right\}  $; thus
$0\leq\cdots\leq$ $j_{u}\leq j_{u+1}\leq\cdots\leq k$ and $\sum_{u=1}^{\ell
}j_{u}=n$. In the other direction given $\left[  j_{i}\right]  _{i=1}^{\ell}$
define $i_{u}=k+u-j_{l+1-u}$ for $1\leq u\leq\ell$. This implies
$i_{u}<i_{u+1}$. Let%
\begin{align*}
F_{u}  &  =\left\{  s\in F^{C}:s>i_{u}\right\}  =\left\{  i_{u}+1,\ldots
,k+\ell\right\}  \backslash\left\{  i_{u+1},\ldots,i_{\ell}\right\}  ,\\
\#F_{u}  &  =\left(  k+\ell-1-i_{u}\right)  -\left(  \ell-u\right)
=k+u-i_{u}=j_{\ell+1-u}.
\end{align*}
Thus $\mathrm{inv}\left(  F\right)  =\sum_{u=1}^{\ell}j_{u}$ , and this proves
the bijection.

In the present situation the set $F=\left\{  Y_{E}\left[  1,i\right]
\right\}  _{i=2}^{N-m-1}$ where $E\in\mathcal{E}_{0}$ (or $2\leq i\leq N-m$
for $E\in\mathcal{E}_{1}$). Suppose the supersymmetric polynomial is in
$\mathcal{M}\left(  \lambda,E\right)  $ (Definition \ref{defMaE}) with
$E\in\mathcal{E}_{0}$; the minimal condition for a column-strict tableau is
$\left\lfloor \lambda,E\right\rfloor \left[  i,1\right]  =i-1$ for $1\leq
i\leq m$ and $\left\lfloor \lambda,E\right\rfloor \left[  1,i\right]
\leq\left\lfloor \lambda,E\right\rfloor \left[  1,i+1\right]  $ for $1\leq
i\leq N-m-1$. Imposing the additional condition that $\left\lfloor
\lambda,E\right\rfloor \left[  1,i\right]  \leq m$ leads to the situation
discussed above, namely a one-to-one correspondence between sets $E$ in
$\mathcal{E}_{0}$ and $\left(  N-m\right)  $-tuples $\left[  0,j_{1}%
,\ldots,j_{N-m-1}\right]  $ such that $\mathrm{inv}\left(  E^{C}\right)
=\sum_{u=1}^{\ell}j_{u}$. There is an analogous statement for $\mathcal{E}%
_{1}$.

\begin{example}
Let $N=8,m=3,\lambda=\left(  3,2,2,2,1,1,1,0\right)  $ and%
\[
\left\lfloor \lambda,E\right\rfloor =%
\begin{bmatrix}
0 & 1 & 1 & 2 & 2\\
\circ & 1 & 2 & 3 &
\end{bmatrix}
\]
then $k=3,\ell=4,\ell~F=\left\{  3+1-2,3+2-2,3+3-1,3+4-1\right\}  =\left\{
2,3,5,6\right\}  $ and $E=\left\{  1,4,7,8\right\}  $. Thus%
\[
Y_{E}=%
\begin{bmatrix}
8 & 6 & 5 & 3 & 2\\
\circ & 7 & 4 & 1 &
\end{bmatrix}
\]
and $\mathrm{inv}\left(  E^{C}\right)  =6$ ($=12-\mathrm{inv}\left(  E\right)
$). Note that $E=E_{S}$, the sink.
\end{example}

The set $E$ arising this way is the sink $E_{S}$ for $\mathcal{E}_{0}$, and
the root $E_{R}$ for $\mathcal{E}_{1}$. There are as many tableaux
$\left\lfloor \lambda,E\right\rfloor $ satisfying $\left\lfloor \lambda
,E\right\rfloor \left[  i,1\right]  =i-1$ for $1\leq i\leq m$ and
$\left\lfloor \lambda,E\right\rfloor \left[  1,i\right]  \leq m$ as generators
of the supersymmetric polynomials (over the ring of symmetric polynomials) in
$\mathcal{P}_{m,0}$; the degrees also match up with the generating function
(\ref{QNmseries}). It seems plausible that these are generators (there is no
nontrivial linear relation with symmetric polynomial coefficients),
considering that multiplying one of them by a nontrivial symmetric polynomial
produces a polynomial containing $x_{1}^{m+1}$ but more is needed to prove
this. In the next section we look at a minimal (in the sense of the dominance
order on partitions) subset of the polynomials.

\section{Norms of certain minimal polynomials\label{mininorm}}

This section concerns the norms of the supersymmetric polynomial
$p_{\lambda,E}\in\mathcal{M}\left(  \lambda,E\right)  $ where $\left[
\lambda,E\right]  \left[  i,1\right]  =i-1$ for $1\leq i\leq m+1$ and $\left[
\lambda,E\right]  \left[  1,j\right]  \leq m$ for $1\leq j\leq M$ (with
$M=N-m$). To facilitate computation with $\left\lfloor \lambda,E_{R}%
\right\rfloor $ use $\mu_{i}$ for row 1, and $\widetilde{\mu}_{i}$ for column
1 as described by%
\begin{equation}%
\begin{bmatrix}
\widetilde{\mu}_{m+1} & \mu_{M-1} & \mu_{M-2} & \ldots & \mu_{3} & \mu_{2} &
\mu_{1}\\
& \widetilde{\mu}_{m} & \widetilde{\mu}_{m-1} & \ldots & \widetilde{\mu}_{2} &
\widetilde{\mu}_{1} &
\end{bmatrix}
. \label{munote}%
\end{equation}
Thus $\left[  \mu_{i}\right]  _{i=1}^{M-1}$ and $\left[  \widetilde{\mu}%
_{i}\right]  _{i=1}^{m+1}$ are both partitions, and the latter is strictly
decreasing. (This is essentially the same as the superpartition notation
except for $\widetilde{\mu}_{m+1}$.) The associated content values are
$c_{i}=M-i$ for $1\leq i\leq M-1$ and $\widetilde{c}_{i}=i-m-1$ for $1\leq
i\leq m+1$. Pick parameters $s,k$ with $0\leq s\leq m-1$ and $1\leq k\leq M-2$
and define $\left\lfloor \lambda,E\right\rfloor $ by $\widetilde{\mu}%
_{i}=m+1-i$ for $1\leq i\leq m+1$, and $\mu_{i}=s+1$ for $1\leq i\leq k$,
$\mu_{i}=s$ for $k+1\leq i\leq M-1$. We will prove a denominator-free formula
for $\left\Vert p_{\lambda,E}\right\Vert ^{2}$. Heuristically this is possible
since $\lambda$ is $\prec$-minimal among $\lambda^{\prime}$ with the same
column 1 and the same sum of entries in row 1 (for example if $M=5$ and the
row-total is $6$ then $\left(  2,2,1,1\right)  $ is $\prec$-minimal). The
formula is%
\begin{align*}
\left\Vert p_{\lambda,E}\right\Vert ^{2}  &  =\ast s!\prod\limits_{i=1}%
^{k-1}\left(  1+i\kappa\right)  \prod\limits_{j=1}^{M-k-2}\left(
1+j\kappa\right)  _{s}\prod\limits_{l=M-k-1}^{M-2}\left(  2+l\kappa\right)
_{s}\\
&  \times\prod\limits_{i=2}^{m}\left(  1-i\kappa\right)  _{i-1}\left(
1-\kappa N\right)  _{m-s-1}\left(  m-s+1-\kappa\left(  m+1\right)  \right)
_{s}\\
&  \times\left(  m-s-\kappa\left(  N-k\right)  \right)  .
\end{align*}
The asterisk $\ast$ stands for $\nu_{k}\dfrac{N\mathcal{C}_{0}^{\left(
k\right)  }\left(  E_{R}\right)  }{\#G_{E_{R}}\mathcal{C}_{1}^{\left(
k\right)  }\left(  E_{S}\right)  }$, the part of the formula independent of
$\kappa$. When $k=0$ so that $\mu_{i}=s$ for $1\leq i\leq M-1$ the formula
reduces to
\[
\ast s!\prod\limits_{j=1}^{M-2}\left(  1+j\kappa\right)  _{s}\prod
\limits_{i=2}^{m}\left(  1-i\kappa\right)  _{i-1}\left(  1-\kappa N\right)
_{m-s}\left(  m-s+1-\kappa\left(  m+1\right)  \right)  _{s}.
\]
There are $\left\vert \widetilde{\mu}\right\vert $ factors of the form
$\left(  a-b\kappa\right)  $ and $\left\vert \mu\right\vert -\mu_{1}$ factors
$\left(  a+b\kappa\right)  $ with $a,b\geq1$. The following are used to
compute telescoping products:

\begin{lemma}
\label{pi0prod}Suppose $u\leq v$, $n\geq1$ and $a$ is arbitrary then%
\[
\prod\limits_{i=u}^{v}\Pi_{0}\left(  n,a+i\right)  =\frac{\left(
1+\kappa\left(  a+u-1\right)  \right)  _{n}\left(  1+\kappa\left(
a+v+1\right)  \right)  _{n-1}}{\left(  1+\kappa\left(  a+u\right)  \right)
_{n-1}\left(  1+\kappa\left(  a+v\right)  \right)  _{n}}.
\]

\end{lemma}

\begin{proof}
The first part of the $i$-product (where $l=n$) equals%
\[
\prod\limits_{i=u}^{v}\left(  \frac{n+\kappa\left(  a+i-1\right)  }%
{n+\kappa\left(  a+i\right)  }\right)  =\frac{n+\kappa\left(  a+u-1\right)
}{n+\kappa\left(  a+v\right)  }%
\]
and the second part $\left(  1\leq l\leq n-1\right)  $ equals%
\begin{align*}
&  \prod\limits_{\ell=1}^{n-1}\prod\limits_{i=u}^{v}\frac{\left(  \ell
+\kappa\left(  a+i-1\right)  \right)  \left(  \ell+\kappa\left(  a+i+1\right)
\right)  }{\left(  \ell+\kappa\left(  a+i\right)  \right)  ^{2}}\\
&  =\prod\limits_{\ell=1}^{n-1}\frac{\left(  \ell+\kappa\left(  a+u-1\right)
\right)  \left(  \ell+\kappa\left(  a+v+1\right)  \right)  }{\left(
\ell+\kappa\left(  a+u\right)  \right)  \left(  \ell+\kappa\left(  A+v\right)
\right)  }\\
&  =\frac{\left(  1+\kappa\left(  a+u-1\right)  \right)  _{n-1}\left(
1+\kappa\left(  a+v+1\right)  \right)  _{n-1}}{\left(  1+\kappa\left(
a+u\right)  \right)  _{n-1}\left(  1+\kappa\left(  a+v\right)  \right)
_{n-1}}..
\end{align*}
Multiply the two parts together.
\end{proof}

\begin{lemma}
\label{prodjj}Let $n\geq1$ then%
\[
\prod\limits_{j=1}^{n}\Pi_{0}\left(  j,-j\right)  =\frac{1}{n\left(
1-\kappa\right)  }\frac{\left(  1-\kappa\left(  n+1\right)  \right)  _{n}%
}{\left(  1-\kappa n\right)  _{n-1}}.
\]

\end{lemma}

\begin{proof}
Proceed by induction. The formula is true for $n=1$ since $\Pi_{0}\left(
1,-1\right)  =\frac{1-2\kappa}{1-\kappa.}$. Suppose the formula is valid for
$n$ and evaluate%
\begin{align*}
\Pi_{0}\left(  n+1,-n-1\right)   &  =\left(  1-\frac{\kappa}{\left(
n+1\right)  \left(  1-\kappa\right)  }\right)  \prod\limits_{l=1}^{n}%
\frac{\left(  l-\kappa n\right)  \left(  l-\kappa\left(  n+2\right)  n\right)
}{\left(  \ell-\kappa\left(  n+1\right)  \right)  ^{2}}\\
&  =\frac{\left(  n+1-\left(  n+2\right)  \kappa\right)  \left(  1-\kappa
n\right)  _{n}\left(  1-\kappa\left(  n+2\right)  \right)  _{n}}{\left(
n+1\right)  \left(  1-\kappa\right)  \left(  1-\kappa\left(  n+1\right)
\right)  _{n}^{2}}\\
&  =\frac{n\left(  1-\kappa n\right)  _{n-1}\left(  1-\kappa\left(
n+2\right)  \right)  _{n+1}}{\left(  n+1\right)  \left(  1-\kappa\left(
n+1\right)  \right)  _{n}^{2}}.
\end{align*}
This telescopes with the product over $1\leq j\leq n$.
\end{proof}

\begin{lemma}
\label{prodFr}Suppose $n\geq1$ and $F\left(  r\right)  $ is a function of $r$
then%
\[
\prod\limits_{r=0}^{n-1}\frac{\left(  1+F\left(  r\right)  \right)  _{n-r}%
}{\left(  1+F\left(  r+1\right)  \right)  _{n-1-r}}=\left(  1+F\left(
0\right)  \right)  _{n}.
\]

\end{lemma}

The formula in Lemma \ref{pi0prod} can be used for $\prod\limits_{i=u}^{v}%
\Pi_{0}\left(  n,a-i\right)  =\prod\limits_{j=-v}^{-u}\Pi_{0}\left(
n,a+j\right)  $.

We split the computation of $\dfrac{\mathcal{P}\left(  \lambda,E_{R}\right)
}{\mathcal{R}_{0}\left(  \lambda^{-},E_{R}\right)  }$ into (1a) $\left(
\lambda_{i^{\prime}},\lambda_{j^{\prime}}\right)  =\left(  \mu_{i},\mu
_{j}\right)  $, (1b) $\left(  \lambda_{i^{\prime}},\lambda_{j^{\prime}%
}\right)  =\left(  s+1,\widetilde{\mu}_{j}\right)  $ and $\widetilde{\mu}%
_{j}>s+1$, (1c) $\left(  \lambda_{i^{\prime}},\lambda_{j^{\prime}}\right)
=\left(  s,\widetilde{\mu}_{j}\right)  $ and $\widetilde{\mu}_{j}>s$; (2a)
$\left(  \lambda_{i^{\prime}},\lambda_{j^{\prime}}\right)  =\left(
\widetilde{\mu}_{i},\widetilde{\mu}_{j}\right)  $, (2b) $\left(
\lambda_{i^{\prime}},\lambda_{j^{\prime}}\right)  =\left(  s+1_{i}%
,\widetilde{\mu}_{j}\right)  $ and $s+1>\widetilde{\mu}_{u}$, (2c) $\left(
\lambda_{i^{\prime}},\lambda_{j^{\prime}}\right)  =\left(  s,\widetilde{\mu
}_{j}\right)  $ and $s>\widetilde{\mu}_{u}$. The partial products are denoted
$P_{1a},P_{1b}$ etc..Lemmas \ref{pi0prod} and \ref{prodFr} are used to
simplify products of $\Pi_{0}$

Part (1a):%
\begin{align*}
P_{1a}  &  =\prod\limits_{i=1}^{k}\left(  1+\left(  M-i\right)  \kappa\right)
_{s+1}\prod\limits_{j=k+1}^{M-1}\left(  1+\left(  M-j\right)  \kappa\right)
_{s}\prod\limits_{i=1}^{k}\prod\limits_{j=k+1}^{M-1}\Pi_{0}\left(
1,j-i\right) \\
&  =\prod\limits_{i=1}^{k}\left(  1+\left(  M-i\right)  \kappa\right)
_{s+1}\prod\limits_{j=k+1}^{M-1}\left(  1+\left(  M-j\right)  \kappa\right)
_{s}\prod\limits_{i=1}^{k}\frac{1+\kappa\left(  k-i\right)  }{1+\kappa\left(
M-1-i\right)  }\\
&  =\prod\limits_{i=1}^{k-1}\left(  1+i\kappa\right)  \prod\limits_{i=1}%
^{k}\left(  2+\kappa\left(  M-i\right)  \right)  _{s}\prod\limits_{j=k+1}%
^{M-1}\left(  1+\kappa\left(  M-j\right)  \right)  _{s}\frac{1+\kappa\left(
M-1\right)  }{1+\kappa\left(  M-k-1\right)  }.
\end{align*}
If $k=0$ then $P_{1a}=\prod\limits_{i=1}^{M-1}\left(  1+\kappa i\right)  _{s}$.

Part (1b): $\mu_{i}=s+1>\widetilde{\mu}_{u}$, then $m+1-u<s+1,u>m-s$, set
$j=M+1-u$%
\begin{align*}
P_{1b}  &  =\prod\limits_{j=0}^{s}\prod\limits_{i=M-k}^{M-1}\Pi_{0}\left(
s+1-j,i+j\right)  =\\
&  =\prod\limits_{j=0}^{s}\frac{\left(  1+\kappa\left(  j+M-k-1\right)
\right)  _{s+1-j}\left(  1+\kappa\left(  j+M\right)  \right)  _{s-j}}{\left(
1+\kappa\left(  j+M-k\right)  \right)  _{s-j}\left(  1+\kappa\left(
j+M-1\right)  \right)  _{s+1-j}}\\
&  =\frac{\left(  1+\kappa\left(  M-k-1\right)  \right)  _{s+1}}{\left(
1+\kappa\left(  M-1\right)  \right)  _{s+1}}.
\end{align*}

Part (1c): $\mu_{i}=s>\widetilde{\mu}_{u}$, set $j=M+1-u$%

\begin{align*}
P_{1c}  &  =\prod\limits_{j=0}^{s-1}\prod\limits_{i=1}^{M-k-1}\Pi_{0}\left(
s-j,i+j\right) \\
&  =\prod\limits_{j=0}^{s-1}\frac{\left(  1+\kappa j\right)  _{s-j}\left(
1+\kappa\left(  j+M-k\right)  \right)  _{s-j-1}}{\left(  1+\kappa\left(
j+1\right)  \right)  _{s-j-1}\left(  1+\kappa\left(  j+M-k-1\right)  \right)
_{s-j}}\\
&  =\frac{\left(  1\right)  _{s}}{\left(  1+\kappa\left(  M-k-1\right)
\right)  _{s}}.
\end{align*}
Then%
\begin{align*}
P_{1a}P_{1b}P_{1c}  &  =\prod\limits_{i=1}^{k-1}\left(  1+i\kappa\right)
\prod\limits_{i=1}^{k}\left(  2+\kappa\left(  M-i\right)  \right)  _{s}%
\prod\limits_{j=k+1}^{M-1}\left(  1+\kappa\left(  M-j\right)  \right)
_{s}\frac{1+\kappa\left(  M-1\right)  }{1+\kappa\left(  M-k-1\right)  }\\
&  \times\frac{s!\left(  s+1+\kappa\left(  M-k-1\right)  \right)  }{\left(
1+\kappa\left(  M-1\right)  \right)  _{s+1}}\\
&  =s!\prod\limits_{i=1}^{k-1}\left(  1+i\kappa\right)  \prod\limits_{j=1}%
^{M-k-2}\left(  1+j\kappa\right)  _{s}\prod\limits_{l=M-k-1}^{M-2}\left(
2+l\kappa\right)  _{s}.
\end{align*}

Part (2a):%
\begin{align*}
P_{2a}  &  =\prod\limits_{i=1}^{m}\left(  1-i\kappa\right)  _{i}%
\prod\limits_{1\leq i<j\leq m+1}\Pi_{0}\left(  j-i,i-j\right) \\
&  =\prod\limits_{i=1}^{m}\left(  1-i\kappa\right)  _{i-1}\left(  m!\left(
1-\kappa\right)  ^{m}\right)  \frac{\left(  1-\left(  m+1\right)
\kappa\right)  _{m}}{m!\left(  1-\kappa\right)  ^{m}}\\
&  =\prod\limits_{i=1}^{m+1}\left(  1-i\kappa\right)  _{i-1}.
\end{align*}
To prove this assume the formula $\prod\limits_{1\leq i<j\leq n+1}\Pi
_{0}\left(  j-i,i-j\right)  =\Pi_{n}=\dfrac{\left(  1-\left(  n+1\right)
\kappa\right)  _{n}}{n!\left(  1-\kappa\right)  ^{n}}$, which is valid for
$n=1$ by Lemma \ref{prodjj} . Then%
\begin{align*}
\frac{\Pi_{n+1}}{\Pi_{n}}  &  =\prod\limits_{i=1}^{n+1}\Pi_{0}\left(
n+2-i,i-n-2\right)  =\prod\limits_{j=1}^{n+1}\Pi_{0}\left(  j,-j\right) \\
&  =\frac{1}{\left(  n+1\right)  \left(  1-\kappa\right)  }\frac{\left(
1-\left(  n+2\right)  \kappa\right)  _{n+1}}{\left(  1-\left(  n+1\right)
\kappa\right)  _{n}}.
\end{align*}
This proves the formula for $\Pi_{n}$ by induction and thus also the formula
for $P_{2a}$.

Part (2b): $\mu_{i}=s+1<\widetilde{\mu}_{u}$, thus $s+1<m+1-u\leq m$ and%
\begin{align*}
P_{2b}  &  =\prod\limits_{j=s+2}^{m}\prod\limits_{i=M-k}^{M-1}\Pi_{0}\left(
j-s-1,-j-i\right) \\
&  =\prod\limits_{j=s+2}^{m}\frac{\left(  1-\kappa\left(  M+j\right)  \right)
_{j-s-1}\left(  1-\kappa\left(  M+j-k-1\right)  \right)  _{j-s-2}}{\left(
1-\kappa\left(  M+j-1\right)  \right)  _{j-s-2}\left(  1-\kappa\left(
M-k+j\right)  \right)  _{j-s-1}}\\
&  =\frac{\left(  1-\kappa\left(  M+m\right)  \right)  _{m-s-1}}{\left(
1-\kappa\left(  M+m-k\right)  \right)  _{m-s-1}}.
\end{align*}

Part (2c): $\mu_{i}=s<\widetilde{\mu}_{u}$, thus $s<m+1-u\leq m$ and (by Lemma
\ref{pi0prod} with $a=-j$, $u=1+k-M$, $v=-1,$and then by Lemma \ref{prodFr})%
\begin{align*}
P_{2c}  &  =\prod\limits_{j=s+1}^{m}\prod\limits_{i=1}^{M-k-1}\Pi_{0}\left(
j-s,-j-i\right) \\
&  =\prod\limits_{j=s+1}^{m}\frac{\left(  1-\kappa\left(  M+j-k\right)
\right)  _{j-s}\left(  1-j\kappa\right)  _{j-s-1}}{\left(  1-\kappa\left(
M+j-k-1\right)  \right)  _{j-s-1}\left(  1-\kappa\left(  j+1\right)  \right)
_{j-s}}\\
&  =\frac{\left(  1-\kappa\left(  M+m-k\right)  \right)  _{m-s}}{\left(
1-\kappa\left(  m+1\right)  \right)  _{m-s}}.
\end{align*}
Then%
\begin{align*}
P_{2a}P_{2b}P_{2c}  &  =\prod\limits_{i=1}^{m+1}\left(  1-i\kappa\right)
_{i-1}\frac{\left(  1-\kappa\left(  M+m\right)  \right)  _{m-s-1}}{\left(
1-\kappa\left(  M+m-k\right)  \right)  _{m-s-1}}\frac{\left(  1-\kappa\left(
M+m-k\right)  \right)  _{m-s}}{\left(  1-\kappa\left(  m+1\right)  \right)
_{m-s}}\\
&  =\prod\limits_{i=1}^{m}\left(  1-i\kappa\right)  _{i-1}\left(
m-s+1-\kappa\left(  m+1\right)  \right)  _{s}\left(  1-\kappa\left(
M+m\right)  \right)  _{m-s-1}\\
&  \times\left(  m-s-\kappa\left(  M+m-k\right)  \right)  .
\end{align*}
This concludes the proof of the formula for $\left\Vert p_{\lambda
,E}\right\Vert ^{2}$, up to a constant independent of $\kappa$.

The content products for the $\left(  N-m,1^{m}\right)  $ situations are
computed with the usual telescoping (with $s\geq1$)%
\begin{align*}
\mathcal{C}_{0}^{\left(  0\right)  }\left(  E_{R}\right)   &  =\frac
{m!}{s!\left(  M+s+1\right)  _{m-s-1}\left(  M+s-k\right)  },\\
\mathcal{C}_{1}^{\left(  0\right)  }\left(  E_{S}\right)   &  =\frac{\left(
M+s+1\right)  _{m-s}\left(  M+s-k\right)  s!}{\left(  m+1\right)  !},\\
\frac{\mathcal{C}_{0}^{\left(  0\right)  }\left(  E_{R}\right)  }%
{\mathcal{C}_{1}^{\left(  0\right)  }\left(  E_{S}\right)  }  &  =\left\{
\frac{m!}{s!\left(  M+s+1\right)  _{m-s-1}\left(  M+s-k\right)  }\right\}
^{2}\frac{m+1}{m+M}.
\end{align*}
If $s=0$ then%
\[
\mathcal{C}_{0}^{\left(  0\right)  }\left(  E_{R}\right)  =\frac{m!}{\left(
M+1\right)  _{m-1}\left(  M-k\right)  },\mathcal{C}_{1}^{\left(  0\right)
}\left(  E_{S}\right)  =\frac{\left(  M+1\right)  _{m}}{\left(  m+1\right)
!}.
\]
Similar calculations can be used to find $\mathcal{C}_{0}^{\left(  1\right)
}\left(  E_{R}\right)  ,\mathcal{C}_{1}^{\left(  1\right)  }\left(
E_{R}\right)  $ for the $\left(  N-m+1,1^{m-1}\right)  $ case (replace $m$ by
$m-1$ in the $\left(  \mu,\widetilde{\mu}\right)  $-notation). The order of
the stabilizer group of $E_{R}$ is $\#G_{E_{R}}=k!\left(  M-k-1\right)  !$
when $s\geq1$ and $k!\left(  M-k\right)  !$ when $s=0$.

\section{Further Results\label{further}}

\subsection{Antisymmetric polynomials}

Recall the duality map $\delta$ from Definition \ref{defdual}. When this map
is applied to $\mathcal{P}_{m,0}$ the NSJP's transform to NSJP's in
$\mathcal{P}_{N-m,1}$ but with the parameter $\kappa$ changed to $-\varkappa$.
We use the notation $\mathcal{D}\left(  \kappa\right)  _{i},\mathcal{U}\left(
\kappa\right)  _{i}$ to indicate the parameter. Apply $\delta$ to
$\mathcal{D}\left(  \kappa\right)  _{i}\sum_{E}p_{E}\left(  x\right)  \phi
_{E}\left(  \theta\right)  $ to obtain%
\begin{align*}
\delta\mathcal{D}\left(  \kappa\right)  _{i}p  &  =\sum_{E}\frac{\partial
}{\partial x_{i}}p_{E}\left(  x\right)  \delta\phi_{E}+\kappa\sum_{j\neq
i}\sum_{E}\frac{p_{E}\left(  x\right)  -p_{E}\left(  x\left(  i,j\right)
\right)  }{x_{i}-x_{j}}\delta\left(  i,j\right)  \phi_{E}\\
&  =\left\{  \sum_{E}\frac{\partial}{\partial x_{i}}p_{E}\left(  x\right)
\delta\phi_{E}-\kappa\sum_{j\neq i}\sum_{E}\frac{p_{E}\left(  x\right)
-p_{E}\left(  x\left(  i,j\right)  \right)  }{x_{i}-x_{j}}\left(  i,j\right)
\delta\phi_{E}\right\} \\
&  =\mathcal{D}\left(  -\kappa\right)  _{i}\delta p.
\end{align*}
Similarly $\delta\mathcal{U}\left(  \kappa\right)  _{i}p=\mathcal{U}\left(
-\kappa\right)  _{i}\delta p$. Suppose $\alpha\in\mathbb{N}_{0}^{N}$ and
$E\in\mathcal{E}_{0}$ then by (\ref{dNJP})%
\[
\delta J_{\alpha,E}\left(  x;\theta\right)  =x^{\alpha}\delta\left(
r_{\alpha}^{-1}T_{E}\right)  +\sum_{\alpha\rhd\beta}x^{\beta}\delta
v_{\alpha,\beta,T}\left(  \kappa;\theta\right)  ,,
\]
and $\delta\left(  r_{\alpha}^{-1}T_{E}\right)  =\sigma\left(  \ell\left(
r_{\alpha}^{-1}\right)  \right)  r_{\alpha}^{-1}\delta T_{E}$ and $\delta
T_{E}=\left(  -1\right)  ^{e}T_{E^{C}}$ (the power of $-1$ can be determined
from Proposition\ref{deltapsi}. Further $\ell\left(  r_{\alpha}^{-1}\right)
=\#\left\{  \left(  i,j\right)  :i<j,w(i)>w(j)\right\}  $. Thus $\delta
J_{\alpha,E}\left(  x,\theta\right)  =\left(  -1\right)  ^{b}J_{\alpha,E^{C}%
}\left(  x;\theta\right)  $, a NSJP for $-\kappa$. The spectral vector stays
the same because $\alpha_{i}+1+\kappa c\left(  r_{\alpha}\left(  i\right)
,E\right)  =\alpha_{i}+1-\kappa c\left(  r_{\alpha}\left(  i\right)
,E^{C}\right)  .$Similar considerations apply to $E\in\mathcal{E}_{1}$.

Suppose $p\left(  x;\theta\right)  =\sum\limits_{E}p_{E}\left(  x\right)
\phi_{E}\left(  \theta\right)  $ is supersymmetric, that is $s_{i}p\left(
x;\theta\right)  =\sum\limits_{E}p_{E}\left(  xs_{i}\right)  s_{i}\phi
_{E}\left(  \theta\right)  =p\left(  x;\theta\right)  $ for $1\leq i<N$ then
$s_{i}\delta p=-\delta s_{i}p=-\delta p$ and $\delta p$ is antisymmetric. Thus
by defining a supersymmetric Jack polynomial $p_{\lambda,E^{C}}$ in
$\mathcal{P}_{N-m,1}$ using the appropriate modification of Theorem
\ref{superPthm} with $-\kappa$ we obtain an antisymmetric polynomial $\delta
p_{\lambda,E^{C}}$ which is an eigenfunction of $\sum_{i=1}^{N}\mathcal{U}%
\left(  \kappa\right)  _{i}^{2}$. Suppose $E\in\mathcal{E}_{0}$ and the
tableau $\left\lfloor \lambda,E\right\rfloor $ is as in Definition
\ref{defMaE} then $\mathcal{M}\left(  \lambda,E\right)  $ contains a unique
(up to a constant multiple) nonzero antisymmetric polynomial if and only if
$\left\lfloor \lambda,E\right\rfloor $ is row-strict (this applies only to row
1, of course; see \cite[Thm. 5.12]{DL2011}).

\begin{example}
let $N=9,m=3,E=\left\{  3,4,6,9\right\}  $and $\lambda=(5,4,3.3.3.2.2.1.0\}$
then%
\begin{align*}
Y_{E^{C}} &  =%
\begin{bmatrix}
9 & 6 & 4 & 3 &  & \\
\circ & 8 & 7 & 5 & 2 & 1
\end{bmatrix}
,\left\lfloor \lambda,E^{C}\right\rfloor =%
\begin{bmatrix}
0 & 2 & 3 & 3 &  & \\
\circ & 1 & 2 & 3 & 4 & 5
\end{bmatrix}
,\\
Y_{E} &  =%
\begin{bmatrix}
9 & 8 & 7 & 5 & 2 & 1\\
\circ & 6 & 4 & 3 &  &
\end{bmatrix}
,\left\lfloor \lambda,E\right\rfloor =%
\begin{bmatrix}
0 & 1 & 2 & 3 & 4 & 5\\
\circ & 2 & 3 & 3 &  &
\end{bmatrix}
.
\end{align*}
Then $p_{\lambda,E^{C}}\in\mathcal{P}_{6,1}$ and is of isotype $\left(
4,1^{5}\right)  $ while $\delta p_{\lambda,E^{C}}$ is antisymmetric and in
$\mathcal{M}\left(  \lambda,E\right)  $, of isotype $\left(  6,1^{3}\right)  $.
\end{example}

\subsection{Wavefunctions on the torus}

This is a sketch of how the polynomials $s\mathcal{P}_{m}$ can be interpreted
on the unit circle. The quantum Calogero-Sutherland model for $N$ identical
particles with $1/r^{2}$ interactions on the circle has the Hamiltonian%
\begin{align*}
\mathcal{H}  &  =\mathcal{-}\sum_{i=1}^{N}\left(  \frac{\partial}{\partial
\phi_{i}}\right)  ^{2}+\frac{1}{2}\sum_{1\leq i<j\leq N}\frac{\kappa\left(
\kappa-1\right)  }{\sin^{2}\left(  \frac{1}{2}\left(  \phi_{i}-\phi
_{j}\right)  \right)  }\\
&  =\sum_{i=1}^{N}\left(  x_{i}\frac{\partial}{\partial x_{i}}\right)
^{2}-2\kappa\sum_{1\leq i<j\leq N}\frac{x_{i}x_{j}\left(  \kappa-1\right)
}{\left(  x_{i}-x_{j}\right)  ^{2}},
\end{align*}
where $x\in\mathbb{T}^{N}$; the torus and its surface measure in terms of
polar coordinates are%
\begin{align*}
\mathbb{T}^{N}  &  :=\left\{  x\in\mathbb{C}^{N}:\left\vert x_{j}\right\vert
=1,1\leq j\leq N\right\}  ,\\
\mathrm{d}m\left(  x\right)   &  =\left(  2\pi\right)  ^{-N}\mathrm{d}\phi
_{1}\cdots\mathrm{d}\phi_{N},~x_{j}=\exp\left(  \mathrm{i}\phi_{j}\right)
,-\pi<\phi_{j}\leq\pi,1\leq j\leq N.
\end{align*}
The time-independent Schr\"{o}dinger equation $\mathcal{H}\psi=E\psi$ has
solutions expressible as the product of the base-state%
\[
\psi_{0}\left(  x\right)  =\prod\limits_{1\leq i<j\leq N}\left\vert 2\sin
\frac{\phi_{i}-\phi_{j}}{2}\right\vert ^{\kappa}=\prod\limits_{1\leq i<j\leq
N}\left(  -\frac{\left(  x_{i}-x_{j}\right)  ^{2}}{x_{i}x_{j}}\right)
^{\kappa/2}%
\]
with Jack polynomials (Lapointe and Vinet \cite{LV1996}). We generalized this
result by introducing a matrix-valued base state \cite{D2017}. In the present
context this is a $\binom{N}{m}\times\binom{N}{m}$-matrix valued function $L$
of $x$ defined on $\mathbb{C}^{N}\backslash\bigcup{}_{i<j}\left\{
x:x_{i}=x_{j}\right\}  $ . Recall the dual basis $\left\{  \widehat{\phi_{E}%
}\right\}  $ defined in \ref{Tnorms} and express $L$ as%
\begin{align*}
L\left(  x\right)   &  =\sum_{E,F}L_{EF}\left(  x\right)  \phi_{E}%
\otimes\widehat{\phi_{F}},\\
L\left(  x\right)  \sum_{F}a_{F}\phi_{F}  &  =\sum_{E}\left(  \sum_{F}%
L_{EF}\left(  x\right)  c_{F}\right)  \phi_{E}.
\end{align*}
Because $\mathcal{P}_{m}$ splits into two irreducible $\mathcal{S}_{N}%
$-modules the matrix $L\left(  x\right)  $ is also split into $L_{0}\left(
x\right)  ,L_{1}\left(  x\right)  $ with%
\begin{align*}
\left(  DM\right)  L_{0}\left(  x\right)   &  =L_{0}\left(  x\right)  \left(
DM\right)  =NL_{0}\left(  x\right)  ,\\
\left(  MD\right)  L_{1}\left(  x\right)   &  =L_{1}\left(  x\right)  \left(
MD\right)  =NL_{1}\left(  x\right)  .
\end{align*}
Let $\gamma_{0}=\dfrac{N-2m-1}{2}$ and $\gamma_{1}=\dfrac{N-2m+1}{2}$, the
average contents for $\left(  N-m,1^{m}\right)  $, $\left(  N-m+1,1^{m-1}%
\right)  $ respectively. The differential system for $L\left(  x\right)  $ is
(for $s=0,1$)%
\[
\frac{\partial}{\partial x_{i}}L_{s}\left(  x\right)  =\kappa L_{s}\left(
x\right)  \left\{  \sum_{j\neq i}\frac{1}{x_{i}-x_{j}}\left(  i,j\right)
-\frac{\gamma_{s}}{x_{i}}I\right\}  ,1\leq i\leq N,
\]
where $L\left(  x\right)  \left(  i,j\right)  =\sum_{E,F}L_{EF}\left(
x\right)  \phi_{E}\otimes\widehat{\left(  i,j\right)  \phi_{F}}$, and $I$ is
the identity matrix. The effect of the term $\frac{\gamma_{s}}{x_{i}}I$ is to
make $L_{s}\left(  x\right)  $ homogeneous of degree $0$, that is,
$L_{s}\left(  cx\right)  =L_{s}\left(  x\right)  $ for $c\in\mathbb{C}%
\backslash\left\{  0\right\}  $: indeed%
\[
\sum_{i=1}^{N}x_{i}\frac{\partial}{\partial x_{i}}L_{s}\left(  x\right)
=\kappa L_{s}\left(  x\right)  \left\{  \sum_{1\leq i<j\leq N}\left(
i,j\right)  -N\gamma_{s}I\right\}  ,
\]
and if $p\left(  \theta\right)  \in\mathcal{P}_{m,0}$ then $\sum_{i<j}\left(
i,j\right)  p\left(  \theta\right)  =\sum_{k=1}^{N}c\left(  k,E_{0}\right)
=\frac{\left(  N-m\right)  \left(  N-m-1\right)  }{2}-\frac{m\left(
m+1\right)  }{2}=\gamma_{0}N$ (and similarly for $\mathcal{P}_{m,1}$). Let
$\mathbb{T}_{reg}^{N}=\mathbb{T}^{N}\backslash\bigcup\limits_{i<j}\left\{
x:x_{i}=x_{j}\right\}  $, then $\mathbb{T}_{reg}^{N}$ has $\left(  N-1\right)
!$ connected components and each component is homotopic to a circle. Let
$x_{0}:=\left(  1,e^{2\pi\mathrm{i}/N},e^{4\pi\mathrm{i}/N},\ldots,e^{2\left(
N-1\right)  \pi\mathrm{i}/N}\right)  $ and denote the connected component of
$\mathbb{T}_{reg}^{N}$ containing $x_{0}$ by $\mathcal{C}_{0}$, called the
\textit{fundamental chamber}. Thus $\mathcal{C}_{0}$ is the set consisting of
$\left(  e^{\mathrm{i}\theta_{1}},\ldots,e^{\mathrm{i}\theta_{N}}\right)  $
with $\theta_{1}<\theta_{2}<\cdots<\theta_{N}<\theta_{1}+2\pi$. The
homogeneity $L\left(  ux\right)  =L\left(  x\right)  $ for $\left\vert
u\right\vert =1$ shows that $L\left(  x\right)  $ has a well-defined analytic
continuation to all of $\mathcal{C}_{0}$ starting from $x_{0}$. Briefly, there
is a method to define $L_{s}\left(  x\right)  $ on the other connected
components of $\mathbb{T}_{reg}^{N}$ so that when restricted to supersymmetric
polynomials$\in$%
\[
L_{s}\left(  x\right)  \sum_{i=1}^{N}\left(  \mathcal{U}_{i}-1-\kappa
\gamma_{s}\right)  ^{2}L_{s}\left(  x\right)  ^{-1}=\sum_{i=1}^{N}\left(
x_{i}\partial_{i}\right)  ^{2}-2\kappa\sum_{1\leq i<j\leq N}\frac{x_{i}%
x_{j}\left(  \kappa-1\right)  }{\left(  x_{i}-x_{j}\right)  ^{2}}%
=\mathcal{H}.
\]
Thus if $p_{\lambda,E}$ is the supersymmetric Jack polynomial associated with
$\left\lfloor \lambda,E\right\rfloor $ where $E\in\mathcal{E}_{0}$ and
$\left\lfloor \lambda,E\right\rfloor $ is column-strict then
\[
\mathcal{H}\left(  L_{0}\left(  x\right)  p_{\lambda,E}\left(  x;\theta
\right)  \right)  =\sum_{i=1}^{N}\left(  \lambda_{i}+\kappa\left(  c\left(
i,E\right)  -\gamma_{0}\right)  \right)  ^{2}L_{0}\left(  x\right)
p_{\lambda,E}\left(  x;\theta\right)  .
\]
This is not the same Hamiltonian defined in \cite{DLM2001}; in that paper the
coupling constant satisfies $\kappa>1$. In terms of the $\left(
\mu,\widetilde{\mu}\right)  $ notation from (\ref{munote}) the eigenvalue is%
\[
\sum_{i=1}^{m+1}\left(  \widetilde{\mu}_{i}+\kappa\left(  i-\frac{N+1}%
{2}\right)  \right)  ^{2}+\sum_{i=1}^{N-m-1}\left(  \mu_{i}+\kappa\left(
\frac{N+1}{2}-i\right)  \right)  ^{2}.
\]
The supersymmetric polynomial of lowest degree has $\mu_{i}=0$ for $1\leq
i\leq N-m-1$ and $\widetilde{\mu}_{i}=m+1-i$ for $1\leq i\leq m+1$ and its
eigenvalue is%
\[
\frac{1}{6}m\left(  m+1\right)  \left\{  \left(  2m+1\right)  \left(
1+\kappa\right)  -3\kappa N\right\}  +\frac{\kappa^{2}}{12}N\left(
N^{2}-1\right)  .
\]
To get the eigenvalues for the isotype $\left(  N-m+1,1^{m-1}\right)  $
replace $m$ by $m-1$ in the formulas.

The adjoint map is defined by $\left\{  \sum_{E}f_{E}\left(  x\right)
\phi_{E}\left(  \theta\right)  \right\}  ^{\ast}\sum_{F}g_{F}\left(  x\right)
\phi_{F}\left(  \theta\right)  =\sum_{E}\overline{f_{E}\left(  x\right)
}g_{E}\left(  x\right)  $. There is a normalization of $L_{s}\left(  x\right)
$ so that for $f,g$%
\[
\int_{\mathbb{T}^{N}}\left\{  L_{s}\left(  x\right)  f\left(  x;\theta\right)
\right\}  ^{\ast}L_{s}\left(  x\right)  g\left(  x;\theta\right)
\mathrm{d}m\left(  x\right)  =\left\langle f,g\right\rangle _{\mathbb{T}},
\]
where $\left\langle f,g\right\rangle _{\mathbb{T}}$ is a conjugate-linear
inner product on $s\mathcal{P}_{m}$ satisfying ($1\leq i\leq N$)%
\begin{align*}
\left\langle wf,wg\right\rangle _{\mathbb{T}}  &  =\left\langle
f,g\right\rangle _{\mathbb{T}},~w\in\emph{S}_{N},\\
\left\langle x_{i}\mathcal{D}_{i}f,g\right\rangle _{\mathbb{T}}  &
=\left\langle f,x_{i}\mathcal{D}_{i}g\right\rangle _{\mathbb{T}},\\
\left\langle x_{i}f,x_{i}g\right\rangle _{\mathbb{T}}  &  =\left\langle
f,g\right\rangle _{\mathbb{T}}.
\end{align*}
This inner product is different from those studied in \cite{DLM2007}. These
properties imply that multiplication by $x_{i}$ is an isometry and that each
$\mathcal{U}_{i}$ is self-adjoint, thus $\left(  \alpha,E\right)  \neq\left(
\beta,F\right)  $ implies $\left\langle J_{\alpha,E},J_{\beta,F}\right\rangle
_{\mathbb{T}}=0$. Also $\left\langle J_{\alpha,E},J_{\alpha,E}\right\rangle
_{\mathbb{T}}=\left\{  \prod\limits_{i=1}^{N}\left(  1+\kappa c\left(
i,E\right)  \right)  _{\alpha_{i}^{+}}\right\}  ^{-1}\left\Vert J_{\alpha
,E}\right\Vert ^{2}$ (from Theorems \ref{Jnorm1} and \ref{Jnorm2}). The
details can be found in \cite{D2017}.


\begin{thebibliography}{99}                                                                                               %


\bibitem {BF1999}T. H. Baker and P. J. Forrester, Symmetric Jack polynomials
from non-symmetric theory, \textit{Ann. Comb. }\textbf{3 }(1999), 159-170.

\bibitem {DLM2001}P. Desrosiers, L. Lapointe, and P. Mathieu, Supersymmetric
Calogero-Moser-Sutherland models and Jack superpolynomials,
\textit{Nucl.Phys.} \textbf{B606} (2001) 547-582.

\bibitem {DLM2003a}P. Desrosiers, L. Lapointe, and P. Mathieu, Jack
superpolynomials, superpartition ordering and determinantal formulas,
\textit{Commun.Math.Phys.} \textbf{233} (2003) 383-402

\bibitem {DLM2003}P. Desrosiers, L. Lapointe, and P. Mathieu, Jack polynomials
in superspace, \textit{Commun. Math. Phys.} \textbf{242} (2003) 331-360

\bibitem {DLM2007}P. Desrosiers, L. Lapointe, and P. Mathieu, Orthogonality of
Jack polynomials in superspace, \textit{Adv. Math.} \textbf{212} (2007) 361-388.

\bibitem {D2017}C. F. Dunkl, Vector-valued Jack polynomials and wavefunctions
on the torus, \textit{J. Phys. A: Math. Theor.} \textbf{50} (2017) 245201 (21pp)

\bibitem {DL2011}C. F. Dunkl and J.-G. Luque, Vector-valued Jack polynomials
from scratch, \textit{SIGMA} \textbf{7} (2011) 26, 48 pp,

\bibitem {G2010}S. Griffeth, Orthogonal functions generalizing Jack
polynomials, \textit{Trans. Amer. Math. Soc.} \textbf{362} (2010), 6131-6157

\bibitem {JK1991}G. James and A. Kerber, \textit{The Representation Theory of
the Symmetric Group}, Encyc. of Math. and its Applic. \textbf{16},
Addison-Wesley, Reading MA, 1981; Cambridge University Press, Cambridge, 2009.

\bibitem {LV1996}L. Lapointe and L. Vinet, Exact operator solution of the
Calogero-Sutherland model, \textit{Comm. Math. Phys.} \textbf{178}, (1996), 425-452.

\bibitem {S1999}R. P. Stanley, \textit{Enumerative Combinatorics}, Vol. 2,
Cambridge Studies in Advanced Mathematics 62, Cambridge Univ. Press, Cambridge 1999.
\end{thebibliography}
\end{document}